\documentclass{birkjour}
\makeindex
\newtheorem{thm}{Theorem}[section]
\newtheorem{lem}[thm]{Lemma}
\newtheorem{prop}[thm]{Proposition}

\newtheorem*{note}{Note}

\newtheorem{pretheorema}{{\bf Theorem}}

\def\Cchi{{\raisebox{.2ex}{\large $\chi$}}}
\title{Determination  of the size of  defining set for Steiner triple systems}
\begin{document}
%----------Author 1
\author[Nazli Besharati]{Nazli Besharati$^{*}$}
\address{Department of Mathematical Sciences,  Payame Noor University,
P.O. Box 19395-3697,   Tehran, I. R. Iran}
\email{nbesharati@pnu.ac.ir\\ {$^*$}Corresponding author}
%\thanks{{$^*$}Corresponding author}
%----------Author 2
\author{M. Mortezaeefar}
\address{Department of Mathematical Sciences, Sharif University of
Technology, P. O. Box 11155-9415, Tehran, I. R. Iran}
\email{ m.mortezaeefar@gmail.com}
%----------classification, keywords, date
%%% ----------------------------------------------------------------------
\maketitle
%%%%%%%%%%%%%%%%%%%%%%
\begin{abstract}
Every  Steiner triple system  is a uniform hypergraph.
The coloring of hypergraph and its special case Steiner triple systems, {STS}$(v)$, is studied extensively.  But the defining set of the coloring of  hypergraph even its special case {STS}$(v)$, is not explored yet.
We study minimum defining set and  the  largest minimal defining set
for $3$-coloring of  {STS}$(v)$.  We   determined
minimum defining set and  the  largest minimal defining set,
for  all  non-isomorphic {STS}$(v)$,  $v\le 15$.
 Also we have found the {\sf defining number} for all Steiner triple systems of order $v$,  and   some lower bounds for the  size of the  largest minimal defining set  for all Steiner triple systems of order $v$,
  for each admissible $v$.
\end{abstract}
%%%%%%%%%%%%%%%%%%%%%%%%%%%%%%%%%%%%%%%%%%%%%%%%%
%
\section{Introduction}\label{Sec:Introduction}
We follow standard notations and concepts from graph theory and design theory. For
these, one may refer to  \cite{MR2368647}, and ~\cite{MR1871828}.

A {\sf $2$-$(v,k,\lambda)$ design} or a  {\sf block design}, for short,  is a pair $(V, \mathcal{B})$
where $V$ is a $v$-set and
$\mathcal{B}$ is a collection of $b$ \  $k$-subsets
of $V$ (blocks) such that, any $2$-subset of $V$ is contained in exactly $\lambda$ blocks.
These conditions imply that  each element of $V$ is contained in constant number,
$r$ blocks.
A $2$-$(v,3,1)$ design is called a {\sf Steiner triple system}, or {\rm STS}$(v)$.

Let $D=(V,\mathcal{B})$ be a Steiner triple system of order $v$,  {\rm STS}$(v)$. A {\sf  (weak) coloring} of $D$ is a mapping $\phi: V \rightarrow C$, where C is a set of
cardinality $m$ whose elements are called colors, such that
$|\phi(B)| >1$ for each $B\in \mathcal{B}$, i.e. there is no
mono-chromatic triple. We say that $D$ is $m$-colorable.
The chromatic number of $D$, $\Cchi(D)$ is the smallest value of $m$ for
which $D$ admits a coloring with $m$ colors. 
A subset of elements  assigned to the same color is called a color class, i.e.
 for each $ c \in C $,  the set
$\phi^{-1}(c) = \{x\in V:\phi(x) = c \}$ is coloring class of color $c$.

An easy counting argument established that only the unique, trivial  {\rm STS}$(3)$ is $2$-chromatic.
By using the Bose and  the Skolem constructions \cite{MR2469212}, we can construct an {\rm STS}$(v)$  with $3$-chromatic. So  for each $ v \equiv 1 , 3 \pmod{6}$, $v \geq 7$, there is  exitly  a Steiner triple system which is  $3$-chromatic.
%

%%%%%%%%%%%%%%%%%%%%%%%%%%%%%%%%%%%%%%
In a given graph $G$, a set $S$ of vertices with an assignment
of colors is called a {\sf defining set} of the vertex coloring of $G$ if there
exists a unique extension of the colors of $S$ to a  $\Cchi(G)$-coloring
of the vertices of $G$. A defining set with minimum cardinality  is called a {\bf \sf smallest defining} set, and its cardinality is the {\sf defining number}, denoted by $d(G,\Cchi)$.
If the coloring of vertices outside a defining set is forced
at each step, we say that  the defining set is {\sf strong}.
On the other hand,  when in each case we have at least one vertex that it has list coloring set
with only one element  the defining set is {\sf strong}, and when
in some steps  we do not have any vertices with list coloring set only one element, the defining set is  {\sf  weak}.
There are some results on defining set of vertex coloring of graphs, such as 
\cite{MR1492638}, \cite{MR1446764}, \cite{MR1674887}, and \cite{MR2194763}.
%\cite{MR2152765}.\cite{MR99k:05137},
The concept of a defining set has been studied, to some extent,
for block designs, and also under another name, a critical set,
for Latin squares. A more general survey of defining sets in combinatorics appears in \cite{MR2011736}.
%%%%%%%%%%%%%%%%%%%%%%%%%%%%%%%%

Erd{\"o}s and Lovasz were apparently the first to consider weak vertex coloring of 
hypergraphs~\cite{MR1178507}.
Later Berge formalized the notion of the weak and strong chromatic number of a 
hypergraph\cite{MR0384579}.
It is clearly that every  Steiner triple system is a uniform hypergraph.
The chromatic  numbers of Steiner  triple  systems  were first studied  by 
Rosa\cite{MR0280390}. 
There are many papers about  the coloring of Steiner  triple  systems, such as
\cite{MR0290993}, \cite{MR1961750} and \cite{MR1953280}. 
But the defining set of the coloring of  hypergraphs  and Steiner triple systems is not explored yet.

Throughout this paper we study $3$-chromatic Steiner triple systems (weak coloring)
 and use a fixed set of colors $C=\{R,G,Y\}$. 
%Let $\phi:V\rightarrow \{R,G,Y\}$.  
Consider  the color classes size with $c_i$, $i=1,2,3$, 
we describe $(c_1,c_2,c_3)$  as a  {\sf coloring pattern} of $3$-coloring of  {\rm STS}$(v)$, such that 
$c_1 \geq c_2 \geq c_3$.

\begin{lem}
{\rm(\cite{MR1961750})}
\label{}
If  $v > 7$,  then  $c_1 \leq r$.
(The parameter $r=\frac{v-1}{2}$ is the
 number of blocks of a Steiner triple system  in which each point appears.)
\end{lem}

%Since every  Steiner triple systems is a uniform hypergraph,
We  introduce the concept of defining set for Steiner triple systems, colorings. 

According to definition of defining set of  vertex coloring of graphs,  
a set of elements with an assignment of colors to them is a
{\sf defining  set} of a Steiner triple system, {\rm STS}$(v)$, if there exists a unique
extension of the colors of this set, to a $3$-coloring of the elements of {\rm STS}$(v)$.
The size of the smallest defining set for an {\rm STS}$(v)$ is called defining number, and 
denoted by $ d({\rm STS}(v),3)$.  The {\sf defining number} for all Steiner triple systems of order $v$ is
\begin{center}
$ d(v,3) = \min \{ d({\rm STS}(v),3)\}.$
\end{center}
%---------------------------------------------------------------------------------------%
The {\sf size of the  largest minimal defining set} for an {\rm STS}$(v)$ is also interesting  and is denoted by ${\mathcal D}({\rm STS}(v),3)$,  and the  size of the  largest minimal defining set  for all Steiner triple systems of order $v$ is
\begin{center}
${\mathcal D}(v,3) = \max \{{\mathcal D}({\rm STS}(v),3)\}$.
\end{center}
We  define the {\sf spectrum} of defining number of all Steiner triple systems of order $v$  as
\begin{center}
$ {\rm spec}(d ({\rm STS}(v), 3))  =\{ n| \ \exists \  {\rm STS}(v), \ d({\rm STS}(v),3)=n\} .$
\end{center}
And  the {\sf spectrum} of the  size of the  largest minimal defining set  for all Steiner triple systems of oder $v$ is
\begin{center}
$ {\rm spec}({\mathcal D} ({\rm STS}(v), 3)) =\{ n| \ \exists \  {\rm STS}(v),  \ {\mathcal D}({\rm STS}(v),3)=n\} .$
\end{center}
%%%%%%%%%%%%%%%%%%
In this paper we study  $ d(v,3)$ and ${\mathcal D}(v,3)$ for each admissible $v$.  Also, we  determined
$ {\rm spec}(d ({\rm STS}(v), 3))$ and $ {\rm spec}({\mathcal D} ({\rm STS}(v), 3)) $,  for  $v\le 15$.
 %%%%%%%%%%%%%%%%%%%%%%%%%%%%%%%%%%%%%%%%%%%%%%%
\section{Defining  sets  for small orders}
Up to isomorphism, there are a unique {\rm STS}$(7)$, a unique  {\rm STS}$(9)$, two {\rm STS}$(13)$s,
and the eighty  {\rm STS}$(15)$s \cite{MR2246267}.

 We have found $ {\rm spec}(d ({\rm STS}(v), 3)) $ and  $ {\rm spec}({\mathcal D}({\rm  STS}(v), 3))$,
 by computer, for  $ 7 \leq v \le 15$. The results are presented  by the Tables in the appendix.  \\ \\
 %%%%%%%%%%%%%%%%%%%%
 $\bullet  \   \  v=7$

Let the unique {${\rm STS}(7)$} be given by $V=\{0,1,\dots, 6\}$ and
the blocks obtained by the mapping $ i\rightarrow (i +1 ) (mod  \ 7)$ on the base block
$B=\{0, 1,3\}$.
The possible coloring patterns are $(c_1, c_2,c_3)=(4,2,1); (3,3,1)$ or$(3,2,2)$ (\cite{MR1961750}).
We determined  the defining set in all coloring patterns, the results  show that
$ d ({\rm STS}(7),3)=d(7,3)=6$, and  ${\mathcal {D}}({\rm STS}(7),3)={\mathcal{D}}(7,3)=6$ 
(See Table \ref{STS7} in appendix). \\ \\
   $  \bullet  \   \  v=9$
   
Let the unique ${\rm STS}(9)$ be given by $V=\{0,1,2,\dots,8\}$, and blocks
${\mathcal{B}}=\{\{0,1,2\}, \{0,3,6\}, \{0,4,8\},\{0,5,7\}, \{3,4,5\}, \{1,4,7\}, \{1,5,6\}, \{1,3,8 \},$ 
 $ \hspace*{1cm} \{6,7,8\} , \{2,5,8\}, \{2,3,7\}, \{2,4,6\} \}.$\\
The possible coloring patterns are $(c_1,c_2,c_3)=(4,4,1); (4,3,2)$ or $(3,3,3)$ (\cite{MR1961750}).
We determined  the defining set in all coloring patterns, the results  show that
$ d ({\rm STS}(9),3) = d(9,3)=7$, and  ${\mathcal {D}}({\rm STS}(9),3)= {\mathcal {D}}(9,3)= 9$
(See Tables \ref{mdefining_sts9} and \ref{Ldefining_sts9} in appendix). \\ \\
   $  \bullet  \   \  v=13$
   
One of the STS(13)s is given by $V =\{0, 1, 2, 3, 4, 5, 6, 7, 8, 9, 10, 11, 12\}$ and  the blocks obtained
by the mapping $ i \rightarrow (i + 1) (mod  \ 13) $ on the two base blocks  $ \{\{0, 1, 4\} , \{0, 2, 7\}\}$.
The second STS(13) is formed by replacing  the blocks $\{0, 1, 4\}, \{0, 2, 7\}, \{2, 4, 9\},
\{7, 9, 1\}$  in the above system by the blocks $\{2, 7, 9\}, \{1, 4, 9\}, \{0, 1, 7\},
\{0, 2, 4\}.$
The possible coloring patterns are $(c_1, c_2, c_3) =(6, 5,2); (6,4,3); (5,5,3)$
or $(5,4,4)$ for both systems (\cite{MR1961750}).
We determined  the defining set  in all coloring patterns for both of them
with computer programming,  and the results  show  that
$ d ({\rm STS}(13),3) = d(13,3)=6$, and ${\mathcal {D}}({\rm STS}(13),3)= {\mathcal {D}}(13,3)= 11$
(See Tables \ref{mdefining  sts13} and  \ref{Ldefining sts13} in appendix). \\ \\
% %%%%%%%%%%%%%%%%%%%%%
 $\bullet  \   \   v=15$

We know, up to isomorphism, there are  eighty  {\rm STS}$(15)$s \cite{MR2246267}.
The possible coloring patterns are $(c_1, c_2, c_3) = (7, 5, 3); (7, 4, 4); (6, 6, 3); (6, 5, 4)$
or $(5, 5, 5).$
 But system ${\sharp 1}$,  (${\sharp 1}$ in the listing  of \cite{MR2246267}),
  has  only  one  coloring pattern $(c_1, c_2, c_3) = (5; 5; 5)$,
and system ${\sharp 7}$  has  only   two coloring patterns $(c_1, c_2, c_3)= (6; 5; 4) or (5; 5; 5) $,
also systems ${\sharp 79}$ and ${\sharp 80}$   have  three coloring patterns  
$(c_1, c_2, c_3) = (6; 6; 3); (6; 5; 4) or (5; 5; 5).$
 (\cite{MR1961750}).
 We determined  the defining  number in all coloring patterns  for all of them,
the results  showed that 
 \begin{displaymath}
d({\rm STS}_{\sharp l}(15),3)=
\left\{ \begin{array}{lll}
7    &  \hspace{.5cm}  if  \  l=1 \\
5    &  \hspace{.5cm}  if   \  l \in \{40, 56, 68\} \\
6    &   \hspace{.5cm}   otherwise
\end{array}\right.
\end{displaymath}
Then   ${\rm spec(d(STS}(15),3))=\{5,6,7\}$, and   $d(15,3)=5$. \\
Also, the results showed that
${\mathcal {D}}({\rm STS}_{\sharp 1}(15),3)=7$,  ${\mathcal {D}}({\rm STS}_{\sharp 7}(15),3)=10$, and for  $ l \in \{74, 77, 80\} $,  ${\mathcal {D}} ({\rm STS}_{\sharp l}(15),3) =  15$.
 The  size of the  largest minimal defining set for
the rest of $ {\rm STS}(15)$  are 11 or 13. \\
 Hence, ${\rm spec}({\mathcal {D}}({\rm STS}(15),3))= \{ 7, 10, 11, 13, 15\}$, and ${\mathcal {D}}(15, 3)= 15$.   As shown in Tables  \ref{defining sts15 NO1} and \ref{defining sts15 NO2} 
for  minimum defining set  and Tables  \ref{Ldefining sts15 NO1} and  \ref{Ldefining sts15 NO2}   for the  largest minimal defining set  in appendix.\\
A summary of results are  presented in the Table \ref{spec(v, 3)}. 
%%%%%%%%%%%%%%%%%%%
%&&&&&&&&&&&&&&&&&&&&&&&
\begin{table}[ht]
\centering
\caption{Spectrums and defining numbers of ${\rm STS}(v)$, for   $ 7 \leq v \le 15$.}
\begin{tabular}{|c|c|c|c|c|}
\hline
$v$                                                        & $7$       & $9$        &   $13$       &   $15$       \\\hline
$ {\rm spec}(d ({\rm STS}(v), 3)) $          & $\{6\}$ & $\{7\}$   & $\{6\}$    &  $\{5,6,7\}$
\\\hline
$ {\rm spec}({\mathcal D}({\rm  STS}(v), 3))$ & $\{6\}$ &$\{9\}$    & $\{11 \}$ & $\{7, 10, 11, 13, 15\}$
\\ \hline
$ d(v, 3)$                                               & $6$       & $7$        & $6$         &  $5$ \\\hline
$ {\mathcal D}(v, 3)$                                     & $6$       & $9$        & $11$        &  $15$ \\\hline
\end{tabular}
\label{spec(v, 3)}
\end{table}
%%%%%%%%%%%%%%%%%%%%%%%%%%%%%%%%%%%%%%%
%\vspace{-.4cm}
\section{General Results}
\begin{thm}
Let {\rm STS}$(v)$, $v >3$, be 3-chromatic, and   $S$ be  the strong defining set for
{\rm STS}$(v)$. Then $\forall v \ \ d({\rm STS}(v),3) \geq 6.$
\end{thm}
\begin{proof}
Suppose we use three colors, red, yellow and green
for coloring of {\rm STS}$(v)$.
To force the  color of an element such  as ${a}$, we must have  a pair of blocks, say $\{a,b,c\}$ and  $\{a,d,e\}$, in which the colors of  elements  $\{b,c,d,e\}$ are given and the color of
$b$ and $c$, are the same, say red, and the color of  $d$ and $e$,
also the same but different from red, say yellow. These will force
the color of $a$ to be green.
Now,  five vertices are colored and four of them, are in the defining set.
Since $\lambda =1$, there is no another block having mono-colored elements from $\{a,b,c,d,e\} $,
but there are blocks with one elements from $\{a,b,c,d,e\} $.
Thus we need to have at least two  elements in the defining set to be able to force the color of a new element.
Therefore, the strong defining set,$S$, has at least six elements.
\end{proof}

%%%%%%%%%%%%%%%%%%%%%%%%%%%%%%%%%%%%%%%%%%%%%%%%%%%%%%%%%%%%
 Rosa and Colbourn introduce the concept of a
 uniquely  colorable Steiner triple system\cite{MR1178507}. This is  defined to be
an $m$-chromatic {\rm STS}$(v)$, in which any
$m$-coloring of the system produces the same partition of the set $V$ into $m$ color classes.

\begin{prop}
\label{1}
$\forall v \geq 25,  \ \ d(v,3)=2.$
\end{prop}
%%%
\begin{proof}
According to uniquely coloring, if Graph (Hypergraph) $G$, has this coloring then $d(G,\Cchi)= \Cchi -1$. Since it is enough to choose one vertex from each $\Cchi -1$ coloring classes, consequently the color of remaining vertices of the graph will be defined. We knew that there exists for every admissible $v\geq 25$,
a uniquely $3$-colorable Steiner triple system~\cite{MR1953280}.
Hence for  every admissible $v\geq 25$, there exists an {\rm STS}$(v)$ with coloring defining number $2$.
Therefore,  $\forall v \geq 25,  \ \ d(v,3)=2.$
\end{proof}
%%%%%%%%%%%%%%

Remark that, according to proposition \ref{1},
If a {\rm STS}$(v)$  has uniquely $3-$colorable  then  its  defining set is weak. 

\begin{thm}
For each $n\geq 1$, there exits an  $L=$ {\rm STS}$(6n+3)$  such that
 ${\mathcal D}(L,3)=6n+3.$ Therefor,  ${\mathcal D}(6n+3,3)=6n+3.$
\end{thm}
\begin{proof}
Considering the Bose construction (\cite{MR2469212}), we  construct one $3$-colorable
${\rm STS}(v)$ for each $v =6n+3$. Note that in the Bose construction 
the elements are  $ V=Q \times \{1,2,3\}$, where $Q=\{1,2,\ldots,2n+1\}$ is an idempotent symmetric Latin square (or  $(Q,\circ)$ is an idempotent commutative quasigroup) of order $(2n+1)$,  and the blocks are defined as
\begin{center}
${\mathcal{B}}= \Big\{\{(x,1),(x,2), (x,3)\}  \ | \ x \in Q \Big\} \ \cup $
$\Big\{ \{(x,i), (y,i), (x \circ y, i+1 \  (mod \ 3)) \} \ | \  x,y \in Q, \  x < y, \ i=1,2,3 \Big\}\cdot $
\end{center}
See Figure \ref{fig: Bose}.
%%%%%%%%%%%%%%%%%%%%%%%%
If we take the second components  of any ordered pairs of $V$, as the color of the elements of blocks, then the pattern of coloring is  $(2n+1, 2n+1, 2n+1)$. On the other hand, according to the Figure \ref{fig: Bose},
 if we use  red color for elements of the first row, yellow  color for elements of  the second row, and
green color for elements of  the third row then we have pattern of
coloring  $(2n+1,2n+1,2n+1)$.
For this patten coloring, the defining set is all of elements of design. Because according to this coloring, if the color of all elements except of one element such as $x$, are defined, then the color of $x$ is not defined (See Figure \ref{fig: Bose}). Therefore,  $V$ is the minimum defining set for this coloring.  
Hence, for any $V=6n+3$ we can construct one ${\rm STS}(v)$ such that the size of largest minimal defining set is $6n+3$, so  ${\mathcal D}(6n+3,3)=6n+3$.
\end{proof}
\begin{figure}[ht]
\begin{center}
\includegraphics[scale=.65]{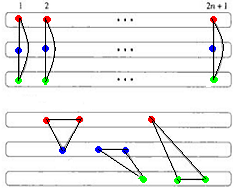}
\caption{The blocks of ${\rm STS}(v)$  are constructed with Bose construction \cite{MR2469212}.}
\label{fig: Bose}
\end{center}
\end{figure}
%---------------------------------------------------------------------------------------%
\begin{thm}
For each $n\geq 1$, there exits an  $M=$ {\rm STS}$(6n+1)$  such that
 ${\mathcal D}(M,3) \ge 5n+1.$ Therefore, ${\mathcal D}(6n+1,3)\ge 5n+1.$
\end{thm}

\begin{proof}
Considering the Skolem construction (\cite{MR2469212}), we  construct one $3$-colorable
${\rm STS}(v)$ for each $v =6n+1$. Note that in  the Skolem construction  the elements are $ V=(Q \times \{1,2,3\}) \cup \{\infty\}$, where $Q=\{1,2,\ldots,2n\}$ is a half  idempotent symmetric Latin square (or $(Q,\circ)$ is a half idempotent commutative quasigroup) of order $2n$,  and the blocks are defined as 
\vspace*{-0.2cm}
\begin{center}
${\mathcal{B}}= \Big\{\{(x,1),(x,2), (x,3)\} \ | \ x \in Q, 1 \leq x \leq n  \Big\} \ \cup$ \\
$\Big\{ \{ \infty, (x+n,i), (x,i+1 \  (mod  \ 3))\} \ | \  x\in Q, \ 1 \leq x \leq n, \ i=1,2,3 \Big\}
\ \cup$ \\
$\Big\{ \{(x,i), (y,i), (x \circ y, i+1 \   ( mod \ 3))\} \ | \  x,y \in Q, \  x < y, \ i=1,2,3 \Big\}\cdot$
\end{center}
See Figure  \ref{fig: skolem}.
  If we color elements of $V$ except $\infty$ similarly to the Bose construction and  take the color of
$\infty$  as one of the three colors, say red,  then
the coloring pattern is  $(2n+1,2n+1,2n+1)$.
If the color of elements of the set 

$ S=\{(x,i) \  | \ i=1,2 ,  x \in Q\} \cup \{(x,3) \ | \ x \in Q,   1 \leq x \leq n \} \cup
\{\infty\}$ \\
is defined,  then according to the defined coloring, the color of the remaining elements
$\{(x+n,3) \ | \ x \in Q, 1 \leq x \leq n\}$ will be uniquely defined.
Hence, $S$ is a defining set and $|S|=5n+1$.
According to the block of design and  Figure~\ref{fig: skolem}, the element $\infty$ always together with two other elements of different color are in a block.
that is why, the color of  $\infty$ is not determined and must  belong to $S$.
Let  $(x,i)$  be any element of  $S$, by the mention of  coloring, there exists  only one restriction  for the coloring of  $(x,i)$ (See Figure~\ref{fig: skolem}). Hence, if  colors of  all elements of $S$ except
$(x,i)$ are known, we can not define the color of  $(x,i)$, therefore  $(x,i)$ should  belong to $S$.
Hence, for any $V=6n+1$ we can construct one ${\rm STS}(v)$ such that the size of largest minimal defining set is $5n+1$,  so ${\mathcal D}(6n+1,3)\ge 5n+1.$
\end{proof}
 %%%%%%%%%%%%%%%%%%%%%%%%
\begin{figure}[ht]
\begin{center}
\includegraphics[scale=.56]{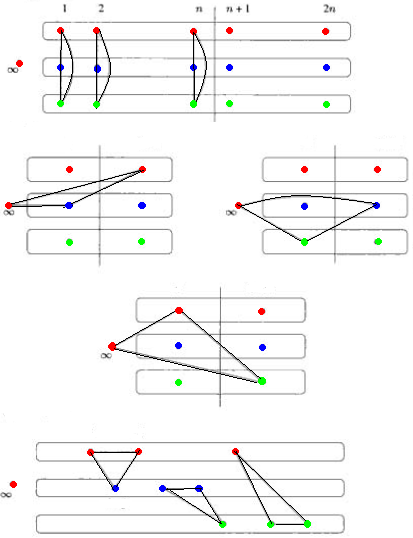}
\caption{The blocks of ${\rm STS}(v)$  are constructed with Skolem construction \cite{MR2469212}.}
\label{fig: skolem}
\end{center}
\end{figure}
%%%%%%%%%%%%%%%%%
%%%%%%%%%%%%%%%%%%%%%%%%%%%%%%%%%%%%%%%%%%%%%%%%%%%%%%%%%%%%%%%%%%%%%%%
%\vspace{-0.5cm}
\section*{Acknowledgement}
The authors  would like to thank  professor E. S. Mahmoodian for suggesting the problem and
very useful comments.
\newpage
%\vspace{-6cm}
\section*{Appendix}
In the tables below, the color of the elements belonging to the defining set, $S$, is indicated by capital letters.
We denote defining number in coloring pattern $p$ by $d_p({\rm STS}(v),3)$. 
With computer programming, we found  minimum and the  largest minimal defining set  in all coloring patterns. \\\\
$\bullet  \   \ {\rm STS}(7)$:
 In the  ${\rm STS}(7)$, minimum defining set and  the  largest minimal defining set occurred in each of coloring pattern.  The results, are shown in Table~\ref{STS7}.  We can easily check that in each of the  cases $S=\{0,1,...,5\}$ is a minimum defining set, and also  the  largest minimal defining set.  \\
%\end{itemize}
 \begin{table}[h!]
\begin{center}
\caption{ Minimum defining set  for  ${\rm STS}(7)$}
\label{STS7}
\begin{tabular}{|c|c||c|c|c|c|c|c|c|}
\hline
 $p=(c_1,c_2,c_3)$  & $d_p({\rm STS}(7),3)$  &$0$ & $1$ & $2$ & $3$ & $4$ & $5$ & $6$  \\\hline
$p=(4,2,1)$ & 6 & $\bf{R}$ & \bf{R}& \bf{R}& \bf{G}& \bf{G}& \bf{R}& ${y}$
\tabularnewline
\hline  
$p=(3,3,1)$ & 6 & \bf{R}& \bf{R}& \bf{R}& \bf{G}& \bf{G}& \bf{G}& ${y}$
\tabularnewline
\hline 
$p=(3,2,2)$ & 6 & \bf{R}& \bf{R}& \bf{R}& \bf{G}& \bf{G}& \bf{Y}& ${y}$     
 \tabularnewline
\hline 
\end{tabular}
\end{center}
\end{table}
\newline 
 $\bullet  \   \ {\rm STS}(9)$:
In the  ${\rm STS}(9)$,  minimum defining set  occurred in  coloring patterns $p=(3,3,2)$ and $p=(3,3,3)$, and the  largest minimal defining set occurred in coloring pattern $p=(3,3,3)$. The results,  are shown in 
Tables~\ref{mdefining_sts9} and \ref{Ldefining_sts9}. \\\\
\vspace{-0.7cm}
\begin{table}[h!]
\begin{center}
\caption{Minimum defining set  for ${\rm STS}(9)$}
\label{mdefining_sts9}
\begin{tabular}{|c|c||c|c|c|c|c|c|c|c|c|c|}
\hline
 $p=(c_1,c_2,c_3)$  & $d_p({\rm STS}(9),3)$  &$0$ & $1$ & $2$ & $3$ & $4$ & $5$ & $6$ & $7$ & $8$  \\\hline
$p=(3,3,2)$ &  7  & \bf{R}& \bf{G}& \bf{G}& \bf{G}& \bf{R}& \bf{R}&  \bf{R}& ${y}$& ${y}$
 \tabularnewline
\hline 
$p=(3,3,3)$ & 7 & \bf{R} & $r$ & $g$ & \bf{R}& \bf{G}& \bf{Y}& \bf{Y}& \bf{G}  &  \bf{Y}  
 \tabularnewline
\hline 
\end{tabular}
\end{center}
\end{table}
%%%%%%%%%%%%%%%%%%%%%%%%%%%%%%%%
\vspace*{-.5cm}
\begin{table}[h!]
\begin{center}
\caption{The  largest minimal defining set  for ${\rm STS}(9)$}
\label{Ldefining_sts9}
\begin{tabular}{|c|c||c|c|c|c|c|c|c|c|c|c|}
\hline
 $p=(c_1,c_2,c_3)$  & ${\mathcal D}({\rm STS}(9),3)$  &$0$ & $1$ & $2$ & $3$ & $4$ & $5$ & $6$ & $7$ & $8$  \\\hline
$p=(3,3,3)$ & 9  & \bf{R}& \bf{G}& \bf{G}& \bf{Y}& \bf{R}& \bf{R}& \bf{G} & \bf{Y}&  \bf{Y}
\tabularnewline
\hline 
\end{tabular}
\end{center}
\end{table}
%%%%%%%%%%%%%%%%%%%%%%%
%%%%%%%%%%%%%%%%%%%%%%%%%%%%%%%%%
\newline
$\bullet  \   \ {\rm STS}(13)$:
In the  ${\rm STS}(13)$,  minimum  and the  largest minimal defining set occurred in  
 coloring pattern $p=(5,4,4)$.  The results,  as shown in 
Tables~\ref{mdefining sts13} and \ref{Ldefining sts13}. \\
\begin{table}[ht]
\caption{Minimum defining set  for ${\rm STS}(13)$}
\label{mdefining sts13}
\begin{tabular}
{|@{\hspace{2pt}}c@{\hspace{2pt}}|c|@{\hspace{2pt}}c@{\hspace{2pt}}||@{\hspace{2pt}}c@{\hspace{4pt}}|c@{\hspace{4pt}}|@{\hspace{4pt}}c@{\hspace{4pt}}|@{\hspace{4pt}}c@{\hspace{2pt}}|@{\hspace{2pt}}c@{\hspace{2pt}}|@{\hspace{2pt}}c@{\hspace{2pt}}|@{\hspace{2pt}}c@{\hspace{2pt}}|@{\hspace{2pt}}c@{\hspace{2pt}}|@{\hspace{2pt}}c@{\hspace{2pt}}|@{\hspace{2pt}}c@{\hspace{2pt}}|@{\hspace{2pt}}c@{\hspace{2pt}}|@{\hspace{2pt}}c@{\hspace{2pt}}|@{\hspace{2pt}}c@{\hspace{2pt}}|@{\hspace{2pt}}c@{\hspace{2pt}}}
\hline
No:   & $(c_1,c_2,c_3)$ & $d_p({\rm STS}(13),3)$
& $0$ & $1$ & $2$ & $3$ & $4$ & $5$ & $6$ & $7$ & $8$ & $9$ &  $10$  &  $11$  &  $12$
\tabularnewline
\hline
${\sharp 1}$ & $(5,4,4)$ & $6$  
& ${y}$& \bf{R}& \bf{G}& ${y}$ & \bf{R} & $g$ & $y$ & \bf{G} & r &  \bf{Y} & $g$ & \bf{R}& ${r}$ \\
\hline
${\sharp 2}$ &  $(5,4,4)$   & $6 $
& \bf{G}& $ \bf{Y}$ & \bf{G}&  \bf{G} & $r$ & \bf{R} & $r $ &  \bf{R} &  ${g}$ &  $y$ & $ y$ & $y$ & $y$
\tabularnewline
\hline
\end{tabular}
\vskip 0.5cm
\caption{The  largest minimal defining set  for ${\rm STS}(13)$}
\label{Ldefining sts13}
\begin{tabular}
{|@{\hspace{1pt}}c@{\hspace{2pt}}|c|@{\hspace{2pt}}c@{\hspace{2pt}}||@{\hspace{2pt}}c@{\hspace{2pt}}|c@{\hspace{2pt}}|@{\hspace{2pt}}c@{\hspace{2pt}}|@{\hspace{2pt}}c@{\hspace{2pt}}|@{\hspace{2pt}}c@{\hspace{2pt}}|@{\hspace{2pt}}c@{\hspace{2pt}}|@{\hspace{2pt}}c@{\hspace{2pt}}|@{\hspace{2pt}}c@{\hspace{2pt}}|@{\hspace{2pt}}c@{\hspace{2pt}}|@{\hspace{2pt}}c@{\hspace{2pt}}|@{\hspace{2pt}}c@{\hspace{2pt}}|@{\hspace{2pt}}c@{\hspace{2pt}}|@{\hspace{2pt}}c@{\hspace{2pt}}|@{\hspace{2pt}}c@{\hspace{1pt}}}
\hline
No:  & $(c_1,c_2,c_3)$ & ${\mathcal{D}}({\rm STS}(13),3)$   
 & $0$ & $1$ & $2$ & $3$ & $4$ & $5$ & $6$ & $7$ & $8$ & $9$ & $10$ & $11$ & $12$
\tabularnewline
\hline
${\sharp 1}$ &  $(5,4,4)$  & $11$ & 
${\bf R}$&  {\bf G} & ${\bf Y}$ & ${\bf Y}$ & ${\bf G}$ & ${\bf Y}$ &   ${\bf G}$ &  ${\bf R}$ &  
${\bf Y}$ &  ${\bf Y}$ &  ${\bf G}$ &  $r$ & $r$ \\
\hline
${\sharp 2}$ &  $(5,4,4)$  & $11$ & 
  ${\bf R}$ & ${\bf G}$ &  ${\bf R}$ & ${\bf R}$ & ${\bf G}$ & ${\bf R}$ & ${\bf G}$  &   ${\bf Y}$ &  ${\bf R}$  &  ${\bf Y}$  & ${\bf G}$  &  $y$ &  $y$
 \tabularnewline
\hline
\end{tabular}
\end{table}
%
%%%%%%%%%%%%%%%%%%%%%%%%%%%%%%%%%%%%
%%%%%%%%%%%%%%%%%%%%%%%%%%%%%%%%%%%
%%%%%%%%%%%%%%%%%%%%%%%%%
\newline
$\bullet  \   \ {\rm STS}(15)$:
Considering the computing results, for all ${\rm STS}(15)$, the smallest minimum defining set and the largest minimal defining set is occorred in  coloring pattern $p=(5,5,5)$. Hence, for minimizing the table, we only mentioned defining set  in this coloring pattern in the  tables. The results,  as shown in 
Tables~\ref{defining sts15 NO1} to \ref{Ldefining sts15 NO2}.\\
\begin{note}
The  results show that, the minimum defining set of {\rm STS}$(7)$, {\rm STS}$(9)$ and {\rm STS}$(13)$ are strong, but the minimum defining set  for some of {\rm STS}$(15)$s are  weak. For example,  the minimum defining set  for  ${\rm STS_{\sharp 40}}(15)$ is weak.
\end{note}

\newpage
\begin{table}[h!]
\caption{Minimum defining set  for ${\rm STS}(15)$}
\label{defining sts15 NO1}
\begin{tabular}{|@{\hspace{3pt}}c@{\hspace{3pt}}|c@{\hspace{3pt}}||@{\hspace{3pt}}c@{\hspace{3pt}}|@{\hspace{3pt}}c@{\hspace{3pt}}|@{\hspace{3pt}}c@{\hspace{3pt}}|c@{\hspace{3pt}}|@{\hspace{3pt}}c@{\hspace{3pt}}|@{\hspace{3pt}}c@{\hspace{3pt}}|@{\hspace{3pt}}c@{\hspace{2pt}}|@{\hspace{2pt}}c@{\hspace{2pt}}|@{\hspace{2pt}}c@{\hspace{2pt}}|@{\hspace{2pt}}c@{\hspace{2pt}}|@{\hspace{2pt}}c@{\hspace{2pt}}|@{\hspace{2pt}}c@{\hspace{2pt}}|@{\hspace{2pt}}c@{\hspace{3pt}}|@{\hspace{3pt}}c@{\hspace{3pt}}|@{\hspace{3pt}}c@{\hspace{3pt}}|@{\hspace{3pt}}c@{\hspace{3pt}}|@{\hspace{3pt}}c@{\hspace{3pt}}|@{\hspace{3pt}}c@{\hspace{3pt}}}
\hline
No:   & $d({\rm STS}(15),3)$  &  
$0$ & $1$ & $2$ & $3$ & $4$ & $5$ & $6$ & $7$ & $8$ & $9$ & $10$ & $11$  & $12$ & $13$  & $14$ 
\tabularnewline
\hline
${\sharp 1}$ &      $ 7 $  &  
$ {\bf R} $ & $ {\bf G} $ & $ {\bf R} $ & $ {\bf G} $ & $ {\bf R} $ & 
$ y $ & $ g $ & $ {\bf G} $ & $ g $ & $ {\bf R} $ & 
$ y $ & $ y $ & $ r $ & $ y $ & $ y $ 
\tabularnewline 
\hline 
${\sharp 2}$  &    $ 6 $  &  
$ {\bf R} $ & $ {\bf G} $ & $ g $ & $ {\bf Y} $ & $ r $ & 
$ {\bf R} $ & $ g $ & $ {\bf Y} $ & $ r $ & $ {\bf R} $ & 
$ g $ & $ g $ & $ y $ & $ y $ & $ y $ 
\tabularnewline 
\hline 
${\sharp 3}$  &      $ 6 $  &  
$ {\bf R} $ & $ {\bf G} $ & $ g $ & $ {\bf Y} $ & $ {\bf G} $ & 
$ y $ & $ y $ & $ {\bf G} $ & $ r $ & $ y $ & 
$ {\bf R} $ & $ y $ & $ r $ & $ r $ & $ g $ 
\tabularnewline 
\hline 
${\sharp 4}$  &    $ 6 $  &  
$ {\bf R} $ & $ {\bf R} $ & $ g $ & $ {\bf G} $ & $ g $ & 
$ {\bf Y} $ & $ r $ & $ {\bf Y} $ & $ r $ & $ y $ & 
$ y $ & $ {\bf R} $ & $ g $ & $ y $ & $ g $ 
\tabularnewline 
\hline 
${\sharp 5}$ &   $ 6 $  &  
$ {\bf R} $ & $ {\bf G} $ & $ g $ & $ {\bf Y} $ & $ r $ & 
$ {\bf R} $ & $ g $ & $ {\bf Y} $ & $ r $ & $ {\bf R} $ & 
$ g $ & $ g $ & $ y $ & $ y $ & $ y $ 
\tabularnewline 
\hline 
${\sharp 6}$  &    $ 6 $  &  
$ {\bf R} $ & $ {\bf R} $ & $ g $ & $ {\bf Y} $ & $ y $ & 
$ r $ & $ {\bf Y} $ & $ {\bf G} $ & $ y $ & $ g $ & 
$ r $ & $ {\bf R} $ & $ g $ & $ y $ & $ g $ 
\tabularnewline 
\hline 
${\sharp 7}$  &    $ 6 $  &  
$ {\bf R} $ & $ r $ & $ g $ & $ {\bf Y} $ & $ {\bf G} $ & 
$ {\bf R} $ & $ g $ & $ {\bf Y} $ & $ y $ & $ y $ & 
$ {\bf R} $ & $ g $ & $ y $ & $ g $ & $ r $ 
\tabularnewline 
\hline 
${\sharp 8}$  &  $ 6 $  &  
$ {\bf R} $ & $ {\bf G} $ & $ {\bf Y} $ & $ g $ & $ {\bf G} $ & 
$ y $ & $ y $ & $ {\bf G} $ & $ r $ & $ y $ & 
$ {\bf R} $ & $ r $ & $ y $ & $ g $ & $ r $ 
\tabularnewline 
\hline 
${\sharp 9}$  &  $ 6 $  &  
$ {\bf R} $ & $ {\bf G} $ & $ {\bf Y} $ & $ g $ & $ g $ & 
$ y $ & $ r $ & $ {\bf R} $ & $ y $ & $ {\bf Y} $ & 
$ {\bf G} $ & $ g $ & $ r $ & $ y $ & $ r $ 
\tabularnewline 
\hline 
${\sharp 10}$  &    $ 6 $  &  
$ {\bf R} $ & $ {\bf R} $ & $ {\bf G} $ & $ g $ & $ r $ & 
$ g $ & $ {\bf Y} $ & $ g $ & $ {\bf R} $ & $ g $ & 
$ y $ & $ {\bf Y} $ & $ r $ & $ y $ & $ y $ 
\tabularnewline 
\hline 
${\sharp 11}$  &   $ 6 $  &  
$ {\bf R} $ & $ {\bf G} $ & $ {\bf Y} $ & $ y $ & $ r $ & 
$ g $ & $ g $ & $ {\bf R} $ & $ y $ & $ {\bf Y} $ & 
$ {\bf G} $ & $ r $ & $ y $ & $ r $ & $ g $ 
\tabularnewline 
\hline 
${\sharp 12}$  &  $ 6 $  &  
$ {\bf R} $ & $ {\bf R} $ & $ {\bf G} $ & $ {\bf Y} $ & $ r $ & 
$ g $ & $ g $ & $ r $ & $ {\bf Y} $ & $ y $ & 
$ y $ & $ {\bf R} $ & $ y $ & $ g $ & $ g $ 
\tabularnewline 
\hline 
${\sharp 13}$  &     $ 6 $  &  
$ {\bf R} $ & $ {\bf G} $ & $ g $ & $ {\bf Y} $ & $ r $ & 
$ {\bf R} $ & $ g $ & $ {\bf R} $ & $ y $ & $ y $ & 
$ r $ & $ {\bf G} $ & $ y $ & $ y $ & $ g $ 
\tabularnewline 
\hline 
${\sharp 14}$  &     $ 6 $  &  
$ {\bf R} $ & $ {\bf G} $ & $ {\bf Y} $ & $ r $ & $ g $ & 
$ y $ & $ {\bf R} $ & $ {\bf G} $ & $ g $ & $ y $ & 
$ {\bf R} $ & $ y $ & $ r $ & $ y $ & $ g $ 
\tabularnewline 
\hline 
${\sharp 15}$  &   $ 6 $  &  
$ {\bf R} $ & $ {\bf G} $ & $ {\bf G} $ & $ y $ & $ y $ & 
$ y $ & $ g $ & $ {\bf R} $ & $ y $ & $ g $ & 
$ {\bf R} $ & $ r $ & $ {\bf G} $ & $ r $ & $ y $ 
\tabularnewline 
\hline 
${\sharp 16}$  &   $ 6 $  &  
$ {\bf R} $ & $ {\bf G} $ & $ {\bf Y} $ & $ r $ & $ {\bf G} $ & 
$ g $ & $ r $ & $ {\bf Y} $ & $ {\bf G} $ & $ y $ & 
$ y $ & $ r $ & $ y $ & $ g $ & $ r $ 
\tabularnewline 
\hline 
${\sharp 17}$  &    $ 6 $  &  
$ {\bf R} $ & $ {\bf G} $ & $ {\bf Y} $ & $ g $ & $ {\bf R} $ & 
$ r $ & $ g $ & $ {\bf Y} $ & $ y $ & $ r $ & 
$ y $ & $ g $ & $ {\bf R} $ & $ y $ & $ g $ 
\tabularnewline 
\hline 
${\sharp 18}$  &     $ 6 $  &  
$ {\bf R} $ & $ {\bf G} $ & $ {\bf Y} $ & $ g $ & $ {\bf G} $ & 
$ y $ & $ r $ & $ {\bf G} $ & $ r $ & $ r $ & 
$ y $ & $ y $ & $ r $ & $ g $ & $ {\bf Y} $ 
\tabularnewline 
\hline 
${\sharp 19}$  &    $ 6 $  &  
$ {\bf R} $ & $ {\bf G} $ & $ y $ & $ {\bf G} $ & $ r $ & 
$ r $ & $ g $ & $ {\bf Y} $ & $ r $ & $ y $ & 
$ {\bf Y} $ & $ g $ & $ r $ & $ {\bf Y} $ & $ g $ 
\tabularnewline 
\hline 
${\sharp 20}$  &   $ 6 $  &  
$ {\bf R} $ & $ {\bf G} $ & $ {\bf Y} $ & $ y $ & $ r $ & 
$ {\bf Y} $ & $ r $ & $ {\bf G} $ & $ g $ & $ y $ & 
$ y $ & $ g $ & $ r $ & $ {\bf R} $ & $ g $ 
\tabularnewline 
\hline 
${\sharp 21}$  & $6$&  ${\bf R}$ & ${\bf G}$ & ${\bf Y} $ & $ {\bf R} $ & $ y $ &  $ r $ & $ y $ & $ g $ & $ r $ & $ {\bf Y} $ & 
$ y $ & $ g $ & $ {\bf G} $ & $ g $ & $ r $ 
\tabularnewline 
\hline 
${\sharp 22}$  &   $ 6 $  &  
$ {\bf R} $ & $ {\bf G} $ & $ {\bf Y} $ & $ {\bf R} $ & $ y $ & 
$ r $ & $ g $ & $ r $ & $ g $ & $ {\bf R} $ & 
$ y $ & $ y $ & $ y $ & $ g $ & $ {\bf G} $ 
\tabularnewline 
\hline 
${\sharp 23}$  &    $ 6 $  &  
$ {\bf R} $ & $ {\bf R} $ & $ {\bf G} $ & $ {\bf Y} $ & $ {\bf G} $ & 
$ {\bf Y} $ & $ r $ & $ r $ & $ g $ & $ y $ & 
$ y $ & $ r $ & $ g $ & $ y $ & $ g $ 
\tabularnewline 
\hline 
${\sharp 24}$  &   $ 6 $  &  
$ {\bf R} $ & $ {\bf G} $ & $ {\bf Y} $ & $ {\bf G} $ & $ r $ & 
$ r $ & $ g $ & $ y $ & $ y $ & $ {\bf Y} $ & 
$ {\bf G} $ & $ g $ & $ r $ & $ y $ & $ r $ 
\tabularnewline 
\hline 
${\sharp 25}$  &     $ 6 $  &  
$ {\bf R} $ & $ {\bf G} $ & $ {\bf Y} $ & $ {\bf Y} $ & $ y $ & 
$ {\bf R} $ & $ g $ & $ r $ & $ y $ & $ g $ & 
$ {\bf G} $ & $ g $ & $ r $ & $ y $ & $ r $ 
\tabularnewline 
\hline 
${\sharp 26}$   &   $ 6 $  &  
$ {\bf R} $ & $ {\bf G} $ & $ {\bf Y} $ & $ {\bf Y} $ & $ r $ & 
$ {\bf R} $ & $ g $ & $ {\bf Y} $ & $ y $ & $ g $ & 
$ y $ & $ g $ & $ r $ & $ r $ & $ g $ 
\tabularnewline 
\hline 
${\sharp 27}$  &    $ 6 $  &  
$ {\bf R} $ & $ {\bf G} $ & $ {\bf Y} $ & $ {\bf Y} $ & $ r $ & 
$ {\bf R} $ & $ g $ & $ g $ & $ g $ & $ y $ & 
$ y $ & $ y $ & $ {\bf R} $ & $ g $ & $ r $ 
\tabularnewline 
\hline 
${\sharp 28}$  &   $ 6 $  &  
$ {\bf R} $ & $ {\bf G} $ & $ {\bf Y} $ & $ {\bf Y} $ & $ {\bf G} $ & 
$ y $ & $ r $ & $ g $ & $ r $ & $ y $ & 
$ r $ & $ g $ & $ r $ & $ {\bf Y} $ & $ g $ 
\tabularnewline 
\hline 
${\sharp 29}$  &  $ 6 $  &  
$ {\bf R} $ & $ {\bf G} $ & $ {\bf Y} $ & $ {\bf Y} $ & $ r $ & 
$ {\bf R} $ & $ g $ & $ {\bf Y} $ & $ y $ & $ r $ & 
$ g $ & $ g $ & $ g $ & $ y $ & $ r $ 
\tabularnewline 
\hline 
${\sharp 30}$   &  $ 6 $  &  
$ {\bf R} $ & $ {\bf G} $ & $ {\bf Y} $ & $ {\bf Y} $ & $ {\bf G} $ & 
$ y $ & $ r $ & $ r $ & $ y $ & $ r $ & 
$ g $ & $ g $ & $ {\bf Y} $ & $ r $ & $ g $ 
\tabularnewline 
\hline 
${\sharp 31}$   &  $ 6 $  &  
$ {\bf R} $ & $ {\bf R} $ & $ {\bf G} $ & $ {\bf G} $ & $ {\bf Y} $ & 
$ g $ & $ r $ & $ {\bf R} $ & $ g $ & $ y $ & 
$ y $ & $ r $ & $ g $ & $ y $ & $ y $ 
\tabularnewline 
\hline 
${\sharp 32}$   &    $ 6 $  &  
$ {\bf R} $ & $ {\bf G} $ & $ {\bf Y} $ & $ {\bf G} $ & $ {\bf Y} $ & 
$ r $ & $ y $ & $ r $ & $ y $ & $ {\bf G} $ & 
$ y $ & $ r $ & $ g $ & $ g $ & $ r $ 
\tabularnewline 
\hline 
${\sharp 33}$   &   $ 6 $  &  
$ {\bf R} $ & $ {\bf G} $ & $ {\bf Y} $ & $ {\bf Y} $ & $ {\bf G} $ & 
$ g $ & $ r $ & $ {\bf R} $ & $ g $ & $ y $ & 
$ y $ & $ y $ & $ r $ & $ g $ & $ r $ 
\tabularnewline 
\hline 
${\sharp 34}$   &   $ 6 $  &  
$ {\bf R} $ & $ {\bf G} $ & $ {\bf Y} $ & $ {\bf Y} $ & $ y $ & 
$ {\bf R} $ & $ g $ & $ g $ & $ g $ & $ y $ & 
$ {\bf R} $ & $ g $ & $ r $ & $ y $ & $ r $ 
\tabularnewline 
\hline 
${\sharp 35}$ &   $ 6 $  &  
$ {\bf R} $ & $ {\bf G} $ & $ {\bf Y} $ & $ {\bf Y} $ & $ {\bf G} $ & 
$ {\bf R} $ & $ y $ & $ g $ & $ g $ & $ r $ & 
$ y $ & $ y $ & $ r $ & $ r $ & $ g $ 
\tabularnewline 
\hline 
${\sharp 36}$   & $ 6 $  &  
$ {\bf R} $ & $ {\bf G} $ & $ {\bf Y} $ & $ {\bf Y} $ & $ r $ & 
$ {\bf R} $ & $ g $ & $ r $ & $ y $ & $ r $ & 
$ g $ & $ {\bf G} $ & $ g $ & $ y $ & $ y $ 
\tabularnewline 
\hline 
${\sharp 37}$   &  $ 6 $  &  
$ {\bf R} $ & $ {\bf G} $ & $ {\bf Y} $ & $ {\bf G} $ & $ g $ & 
$ y $ & $ {\bf Y} $ & $ r $ & $ y $ & $ {\bf R} $ & 
$ g $ & $ g $ & $ r $ & $ y $ & $ r $ 
\tabularnewline 
\hline 
${\sharp 38}$  &    $ 6 $  &  
$ {\bf R} $ & $ {\bf G} $ & $ {\bf Y} $ & $ {\bf Y} $ & $ {\bf G} $ & 
$ g $ & $ y $ & $ g $ & $ {\bf R} $ & $ r $ & 
$ y $ & $ y $ & $ r $ & $ g $ & $ r $ 
\tabularnewline 
\hline 
${\sharp 39}$   &   $ 6 $  &  
$ {\bf R} $ & $ {\bf G} $ & $ {\bf Y} $ & $ {\bf G} $ & $ {\bf Y} $ & 
$ r $ & $ y $ & $ {\bf R} $ & $ g $ & $ g $ & 
$ y $ & $ r $ & $ y $ & $ r $ & $ g $ 
\tabularnewline 
\hline 
${\sharp 40}$   &    $ 5 $  &  
$ r $ & $ r $ & $ g $ & $ {\bf G} $ & $ y $ & 
$ g $ & $ {\bf R} $ & $ r $ & $ {\bf Y} $ & $ y $ & 
$ y $ & $ r $ & $ {\bf Y} $ & $ g $ & $ {\bf G} $ 
\tabularnewline 
\hline 
\end{tabular}
\end{table}
%%%%%%%%%%%%%%%%%%%%%%%%%%
\begin{table}[h!]
\caption{Minimum defining set  for ${\rm STS}(15)$}
\label{defining sts15 NO2}
\begin{tabular}{|@{\hspace{3pt}}c@{\hspace{3pt}}|c@{\hspace{3pt}}||@{\hspace{3pt}}c@{\hspace{3pt}}|@{\hspace{3pt}}c@{\hspace{3pt}}|@{\hspace{3pt}}c@{\hspace{3pt}}|c@{\hspace{3pt}}|@{\hspace{3pt}}c@{\hspace{3pt}}|@{\hspace{3pt}}c@{\hspace{3pt}}|@{\hspace{3pt}}c@{\hspace{2pt}}|@{\hspace{2pt}}c@{\hspace{2pt}}|@{\hspace{2pt}}c@{\hspace{2pt}}|@{\hspace{2pt}}c@{\hspace{2pt}}|@{\hspace{2pt}}c@{\hspace{2pt}}|@{\hspace{2pt}}c@{\hspace{2pt}}|@{\hspace{2pt}}c@{\hspace{3pt}}|@{\hspace{3pt}}c@{\hspace{3pt}}|@{\hspace{3pt}}c@{\hspace{3pt}}|@{\hspace{3pt}}c@{\hspace{3pt}}|@{\hspace{3pt}}c@{\hspace{3pt}}|@{\hspace{3pt}}c@{\hspace{3pt}}}
\hline
No:  &  $d({\rm STS}(15),3)$  & 
$0$ & $1$ & $2$ & $3$ & $4$ & $5$ & $6$ & $7$ & $8$ & $9$ & $10$ & $11$  & $12$ & $13$  & $14$ 
\tabularnewline
\hline
${\sharp 41}$ &  $ 6 $  &  
$ {\bf R} $ & $ {\bf G} $ & $ {\bf Y} $ & $ {\bf Y} $ & $ {\bf G} $ & 
$ g $ & $ r $ & $ g $ & $ {\bf R} $ & $ r $ & 
$ g $ & $ y $ & $ r $ & $ y $ & $ y $ 
\tabularnewline
\hline
${\sharp 42}$  &   $ 6 $  &  
$ {\bf R} $ & $ {\bf G} $ & $ {\bf Y} $ & $ {\bf Y} $ & $ {\bf G} $ & 
$ g $ & $ y $ & $ r $ & $ g $ & $ r $ & 
$ y $ & $ {\bf Y} $ & $ r $ & $ g $ & $ r $ 
\tabularnewline 
\hline 
${\sharp 43}$   &   $ 6 $  &  
$ {\bf R} $ & $ {\bf G} $ & $ {\bf Y} $ & $ {\bf Y} $ & $ {\bf G} $ & 
$ {\bf R} $ & $ y $ & $ r $ & $ y $ & $ g $ & 
$ g $ & $ y $ & $ r $ & $ r $ & $ g $ 
\tabularnewline 
\hline 
${\sharp 44}$  &      $ 6 $  &  
$ {\bf R} $ & $ {\bf G} $ & $ {\bf Y} $ & $ {\bf Y} $ & $ {\bf G} $ & 
$ {\bf R} $ & $ y $ & $ g $ & $ g $ & $ r $ & 
$ y $ & $ y $ & $ r $ & $ r $ & $ g $ 
\tabularnewline 
\hline 
${\sharp 45}$ &    $ 6 $  &  
$ {\bf R} $ & $ {\bf G} $ & $ {\bf Y} $ & $ {\bf G} $ & $ {\bf Y} $ & 
$ y $ & $ {\bf R} $ & $ y $ & $ r $ & $ g $ & 
$ g $ & $ r $ & $ y $ & $ g $ & $ r $ 
\tabularnewline 
\hline 
${\sharp 46}$   &  $ 6 $  &  
$ {\bf R} $ & $ {\bf G} $ & $ {\bf Y} $ & $ {\bf Y} $ & $ {\bf G} $ & 
$ {\bf R} $ & $ y $ & $ g $ & $ y $ & $ r $ & 
$ g $ & $ y $ & $ r $ & $ g $ & $ r $ 
\tabularnewline 
\hline 
${\sharp 47}$   &     $ 6 $  &  
$ {\bf R} $ & $ {\bf G} $ & $ {\bf Y} $ & $ {\bf Y} $ & $ {\bf G} $ & 
$ {\bf R} $ & $ y $ & $ r $ & $ y $ & $ g $ & 
$ g $ & $ y $ & $ r $ & $ g $ & $ r $ 
\tabularnewline 
\hline 
${\sharp 48}$   &    $ 6 $  &  
$ {\bf R} $ & $ {\bf G} $ & $ {\bf Y} $ & $ {\bf Y} $ & $ r $ & 
$ {\bf R} $ & $ g $ & $ g $ & $ y $ & $ r $ & 
$ g $ & $ r $ & $ g $ & $ {\bf Y} $ & $ y $ 
\tabularnewline 
\hline 
${\sharp 49}$   &   $ 6 $  &  
$ {\bf R} $ & $ {\bf G} $ & $ {\bf Y} $ & $ {\bf G} $ & $ y $ & 
$ y $ & $ {\bf G} $ & $ {\bf R} $ & $ y $ & $ g $ & 
$ r $ & $ r $ & $ g $ & $ r $ & $ y $ 
\tabularnewline 
\hline 
${\sharp 50}$  &   $ 6 $  &  
$ {\bf R} $ & $ {\bf G} $ & $ {\bf Y} $ & $ {\bf G} $ & $ {\bf Y} $ & 
$ y $ & $ {\bf R} $ & $ y $ & $ r $ & $ g $ & 
$ g $ & $ r $ & $ y $ & $ g $ & $ r $ 
\tabularnewline 
\hline 
${\sharp 51}$  &    $ 6 $  &  
$ {\bf R} $ & $ {\bf G} $ & $ {\bf Y} $ & $ {\bf Y} $ & $ {\bf G} $ & 
$ {\bf R} $ & $ y $ & $ r $ & $ g $ & $ g $ & 
$ y $ & $ y $ & $ r $ & $ g $ & $ r $ 
\tabularnewline 
\hline 
${\sharp 52}$  &     $ 6 $  &  
$ {\bf R} $ & $ {\bf G} $ & $ {\bf Y} $ & $ {\bf Y} $ & $ {\bf G} $ & 
$ y $ & $ y $ & $ r $ & $ g $ & $ g $ & 
$ {\bf R} $ & $ r $ & $ g $ & $ y $ & $ r $ 
\tabularnewline 
\hline 
${\sharp 53}$  &    $ 6 $  &  
$ {\bf R} $ & $ {\bf G} $ & $ {\bf Y} $ & $ g $ & $ {\bf R} $ & 
$ y $ & $ {\bf Y} $ & $ {\bf R} $ & $ y $ & $ r $ & 
$ g $ & $ g $ & $ r $ & $ y $ & $ g $ 
\tabularnewline 
\hline 
${\sharp 54}$  &      $ 6 $  &  
$ {\bf R} $ & $ {\bf G} $ & $ {\bf Y} $ & $ {\bf Y} $ & $ y $ & 
$ {\bf R} $ & $ y $ & $ g $ & $ r $ & $ r $ & 
$ g $ & $ r $ & $ {\bf Y} $ & $ g $ & $ g $ 
\tabularnewline 
\hline 
${\sharp 55}$   &   $ 6 $  &  
$ {\bf R} $ & $ {\bf G} $ & $ {\bf Y} $ & $ {\bf Y} $ & $ {\bf G} $ & 
$ {\bf R} $ & $ y $ & $ g $ & $ g $ & $ r $ & 
$ y $ & $ y $ & $ r $ & $ g $ & $ r $ 
\tabularnewline 
\hline 
${\sharp 56}$   &   $ 5 $  &  
$ r $ & $ r $ & $ g $ & $ y $ & $ {\bf Y} $ & 
$ y $ & $ r $ & $ {\bf G} $ & $ {\bf G} $ & $ r $ & 
$ g $ & $ {\bf Y} $ & $ g $ & $ y $ & $ {\bf R} $ 
\tabularnewline 
\hline 
${\sharp 57}$  &    $ 6 $  &  
$ {\bf R} $ & $ {\bf G} $ & $ {\bf Y} $ & $ {\bf Y} $ & $ {\bf G} $ & 
$ {\bf R} $ & $ y $ & $ g $ & $ g $ & $ r $ & 
$ y $ & $ y $ & $ r $ & $ g $ & $ r $ 
\tabularnewline 
\hline 
${\sharp 58}$  &     $ 6 $  &  
$ {\bf R} $ & $ {\bf G} $ & $ {\bf Y} $ & $ {\bf G} $ & $ {\bf Y} $ & 
$ y $ & $ y $ & $ r $ & $ {\bf G} $ & $ g $ & 
$ r $ & $ g $ & $ r $ & $ y $ & $ r $ 
\tabularnewline 
\hline 
 ${\sharp 59}$  &    $ 6 $  &  
$ {\bf R} $ & $ {\bf G} $ & $ {\bf Y} $ & $ {\bf G} $ & $ {\bf Y} $ & 
$ y $ & $ {\bf R} $ & $ g $ & $ y $ & $ r $ & 
$ y $ & $ g $ & $ r $ & $ r $ & $ g $ 
\tabularnewline 
\hline 
${\sharp 60}$  &    $ 6 $  &  
$ {\bf R} $ & $ {\bf G} $ & $ {\bf Y} $ & $ {\bf Y} $ & $ {\bf G} $ & 
$ {\bf R} $ & $ y $ & $ g $ & $ g $ & $ y $ & 
$ r $ & $ r $ & $ g $ & $ r $ & $ y $ 
\tabularnewline 
\hline 
${\sharp 61}$   &    $ 6 $  &  
$ {\bf R} $ & $ {\bf G} $ & $ {\bf Y} $ & $ {\bf Y} $ & $ r $ & 
$ y $ & $ r $ & $ g $ & $ y $ & $ r $ & 
$ g $ & $ g $ & $ r $ & $ {\bf Y} $ & $ {\bf G} $ 
\tabularnewline 
\hline 
${\sharp 62}$ &   $ 6 $  &  
$ {\bf R} $ & $ {\bf G} $ & $ {\bf Y} $ & $ {\bf R} $ & $ y $ & 
$ y $ & $ r $ & $ r $ & $ g $ & $ g $ & 
$ g $ & $ {\bf R} $ & $ y $ & $ y $ & $ {\bf G} $ 
\tabularnewline 
\hline 
${\sharp 63}$  &    $ 6 $  &  
$ {\bf R} $ & $ {\bf G} $ & $ {\bf Y} $ & $ {\bf G} $ & $ y $ & 
$ r $ & $ g $ & $ r $ & $ {\bf G} $ & $ r $ & 
$ y $ & $ r $ & $ y $ & $ {\bf Y} $ & $ g $ 
\tabularnewline 
\hline 
${\sharp 64}$  &    $ 6 $  &  
$ {\bf R} $ & $ {\bf G} $ & $ {\bf Y} $ & $ {\bf R} $ & $ y $ & 
$ r $ & $ {\bf G} $ & $ r $ & $ g $ & $ {\bf G} $ & 
$ r $ & $ y $ & $ y $ & $ g $ & $ y $ 
\tabularnewline 
\hline 
${\sharp 65}$   &   $ 6 $  &  
$ {\bf R} $ & $ {\bf G} $ & $ {\bf Y} $ & $ y $ & $ {\bf R} $ & 
$ {\bf G} $ & $ g $ & $ r $ & $ y $ & $ r $ & 
$ g $ & $ {\bf Y} $ & $ r $ & $ y $ & $ g $ 
\tabularnewline 
\hline 
${\sharp 66}$  &  $ 6 $  &  
$ {\bf R} $ & $ {\bf G} $ & $ {\bf Y} $ & $ {\bf R} $ & $ y $ & 
$ {\bf G} $ & $ g $ & $ {\bf G} $ & $ y $ & $ r $ & 
$ y $ & $ g $ & $ r $ & $ y $ & $ r $ 
\tabularnewline 
\hline 
${\sharp 67}$    &    $ 6 $  &  
$ {\bf R} $ & $ {\bf G} $ & $ {\bf Y} $ & $ {\bf R} $ & $ y $ & 
$ r $ & $ {\bf G} $ & $ {\bf G} $ & $ g $ & $ r $ & 
$ y $ & $ r $ & $ y $ & $ y $ & $ g $ 
\tabularnewline 
\hline 
${\sharp 68}$   &  $ 5 $  &  
$ {\bf R} $ & $ r $ & $ g $ & $ g $ & $ y $ & 
$ r $ & $ y $ & $ r $ & $ g $ & $ {\bf Y} $ & 
$ {\bf G} $ & $ r $ & $ y $ & $ {\bf G} $ & $ {\bf Y} $ 
\tabularnewline 
\hline 
${\sharp 69}$    &    $ 6 $  &  
$ {\bf R} $ & $ {\bf G} $ & $ {\bf Y} $ & $ {\bf R} $ & $ y $ & 
$ {\bf Y} $ & $ {\bf G} $ & $ y $ & $ r $ & $ g $ & 
$ g $ & $ r $ & $ y $ & $ r $ & $ g $ 
\tabularnewline 
\hline 
${\sharp 70}$     &    $ 6 $  &  
$ {\bf R} $ & $ {\bf G} $ & $ {\bf Y} $ & $ {\bf G} $ & $ {\bf Y} $ & 
$ r $ & $ g $ & $ y $ & $ r $ & $ y $ & 
$ y $ & $ r $ & $ g $ & $ g $ & $ {\bf R} $ 
\tabularnewline 
\hline 
${\sharp 71}$  &    $ 6 $  &  
$ {\bf R} $ & $ {\bf G} $ & $ {\bf Y} $ & $ {\bf G} $ & $ {\bf Y} $ & 
$ y $ & $ y $ & $ g $ & $ r $ & $ r $ & 
$ {\bf G} $ & $ g $ & $ r $ & $ r $ & $ y $ 
\tabularnewline 
\hline 
${\sharp 72}$   &   $ 6 $  &  
$ {\bf R} $ & $ {\bf G} $ & $ {\bf Y} $ & $ {\bf Y} $ & $ {\bf G} $ & 
$ {\bf R} $ & $ y $ & $ g $ & $ g $ & $ r $ & 
$ g $ & $ r $ & $ y $ & $ y $ & $ r $ 
\tabularnewline 
\hline 
${\sharp 73}$  &    $ 6 $  &  
$ {\bf R} $ & $ {\bf G} $ & $ {\bf Y} $ & $ {\bf G} $ & $ {\bf Y} $ & 
$ y $ & $ {\bf R} $ & $ y $ & $ r $ & $ g $ & 
$ g $ & $ r $ & $ g $ & $ y $ & $ r $ 
\tabularnewline 
\hline 
${\sharp 74}$    & $ 6 $  &  
$ {\bf R} $ & $ {\bf G} $ & $ {\bf Y} $ & $ {\bf G} $ & $ {\bf G} $ & 
$ r $ & $ y $ & $ g $ & $ y $ & $ y $ & 
$ r $ & $ {\bf Y} $ & $ r $ & $ g $ & $ r $ 
\tabularnewline 
\hline 
${\sharp 75}$     &   $ 6 $  &  
$ {\bf R} $ & $ {\bf G} $ & $ {\bf Y} $ & $ {\bf G} $ & $ {\bf Y} $ & 
$ r $ & $ g $ & $ r $ & $ {\bf Y} $ & $ g $ & 
$ r $ & $ y $ & $ g $ & $ r $ & $ y $ 
\tabularnewline 
\hline 
${\sharp 76}$      &   $ 6 $  &  
$ {\bf R} $ & $ {\bf G} $ & $ {\bf Y} $ & $ {\bf Y} $ & $ {\bf G} $ & 
$ r $ & $ y $ & $ r $ & $ g $ & $ g $ & 
$ {\bf Y} $ & $ r $ & $ g $ & $ y $ & $ r $ 
\tabularnewline 
\hline 
${\sharp 77}$      &    $ 6 $  &  
$ {\bf R} $ & $ {\bf G} $ & $ {\bf Y} $ & $ y $ & $ {\bf R} $ & 
$ {\bf R} $ & $ g $ & $ r $ & $ {\bf G} $ & $ g $ & 
$ y $ & $ r $ & $ g $ & $ y $ & $ y $ 
\tabularnewline 
\hline 
${\sharp 78}$      &    $ 6 $  &  
$ {\bf R} $ & $ {\bf G} $ & $ {\bf Y} $ & $ {\bf R} $ & $ {\bf G} $ & 
$ r $ & $ y $ & $ g $ & $ g $ & $ y $ & 
$ {\bf R} $ & $ y $ & $ y $ & $ r $ & $ g $ 
\tabularnewline 
\hline 
${\sharp 79}$     & $ 6 $  &  
$ {\bf R} $ & $ {\bf G} $ & $ {\bf Y} $ & $ {\bf Y} $ & $ {\bf G} $ & 
$ r $ & $ y $ & $ g $ & $ r $ & $ y $ & 
$ r $ & $ g $ & $ {\bf G} $ & $ y $ & $ r $ 
\tabularnewline 
\hline 
${\sharp80}$   &     $ 6 $  &  
$ {\bf R} $ & $ {\bf G} $ & $ {\bf Y} $ & $ {\bf R} $ & $ y $ & 
$ {\bf Y} $ & $ r $ & $ {\bf R} $ & $ g $ & $ g $ & 
$ g $ & $ y $ & $ y $ & $ r $ & $ g $ 
\tabularnewline 
\hline 
\end{tabular}
\end{table}
%
%%%%%%%%%%%%%%%%%%%%%%%%%%%%%
\begin{table}[h!]
\caption{The  largest minimal defining set of ${\rm STS}(15)$}
\label{Ldefining sts15 NO1}
\begin{tabular}{|@{\hspace{3pt}}c@{\hspace{3pt}}|c@{\hspace{3pt}}||@{\hspace{3pt}}c@{\hspace{3pt}}|@{\hspace{3pt}}c@{\hspace{3pt}}|@{\hspace{3pt}}c@{\hspace{3pt}}|c@{\hspace{3pt}}|@{\hspace{3pt}}c@{\hspace{3pt}}|@{\hspace{3pt}}c@{\hspace{3pt}}|@{\hspace{3pt}}c@{\hspace{2pt}}|@{\hspace{2pt}}c@{\hspace{2pt}}|@{\hspace{2pt}}c@{\hspace{2pt}}|@{\hspace{2pt}}c@{\hspace{2pt}}|@{\hspace{2pt}}c@{\hspace{2pt}}|@{\hspace{2pt}}c@{\hspace{2pt}}|@{\hspace{2pt}}c@{\hspace{3pt}}|@{\hspace{3pt}}c@{\hspace{3pt}}|@{\hspace{3pt}}c@{\hspace{3pt}}|@{\hspace{3pt}}c@{\hspace{3pt}}|@{\hspace{3pt}}c@{\hspace{3pt}}|@{\hspace{3pt}}c@{\hspace{3pt}}}
\hline
 No: &   ${\mathcal{D}}({\rm STS}(15),3)$   &   
$0$ & $1$ & $2$ & $3$ & $4$ & $5$ & $6$ & $7$ & $8$ & $9$ & $10$ & $11$  & $12$ & $13$  & $14$
\tabularnewline
\hline
${\sharp1}$   &  $ 7 $  &
$ {\bf R} $ & $ {\bf G} $ & $ {\bf G} $ & $ {\bf G} $ & $ {\bf Y} $ &
$ {y} $ & $ {y} $ & $ {\bf G} $ & $ r $ & $ {\bf Y} $ & $ r $ & $ r $ & $ g $ & $ {y} $ & $ r $
\tabularnewline
\hline
${\sharp 2}$   & $ 11 $  &
$ {\bf R} $ & $ {\bf G} $ & $ r$ & $ {\bf G} $ & $r$ & $ {y} $ & $r $ & $ {\bf G} $ & $  {\bf Y}  $ & $ {\bf Y} $ &  $  {\bf G} $ & $  {\bf Y}  $ & $  {\bf G}  $ & $ {\bf G} $ & $ {\bf Y} $
\tabularnewline
\hline
${\sharp 3}$   &   $ 11 $  &
$ {\bf R} $ & $ {\bf G} $ & $ r$ & $ {\bf G} $ & $r$ &
$ {y} $ & $r $ & $ {\bf G} $ & $  {\bf Y}  $ & $ {\bf Y} $ &  $  {\bf G} $ & $  {\bf Y}  $ & $  {\bf G}  $ & $ {\bf G} $ & $ {\bf Y}  $
 \tabularnewline
\hline      
${\sharp 4}$    & $ 11 $  &
$ {\bf R} $ & $ {\bf G} $ & $ {\bf Y}$ & $ {\bf R} $ & $ {\bf Y}$ &
$ {\bf G}$  & $r $ & $ {\bf G} $ & $  {\bf G}  $ & $ {\bf R} $ &  ${y} $ & $  {\bf Y}  $ & $  {\bf G}  $ & ${y} $ & $ r  $
\tabularnewline
\hline      
${\sharp 5}$ &    $ 11 $  &
$ {\bf R} $ & $ {\bf G} $ & $r$ & $ {\bf G} $ & $ {\bf Y}$ &
$ {\bf Y}$  & ${\bf G} $ & $ {\bf G} $ & $r$ & ${y}$ &   $r$ & $  {\bf Y}  $ & $  {\bf G}  $ & ${\bf G} $ & ${\bf Y} $
\tabularnewline
\hline      
${\sharp 6}$   &    $ 11 $  &   
$ {\bf R} $ & $ {\bf G} $ & $ {\bf R} $ & $ {\bf R} $ & $ {\bf G}$ &
$ {\bf R}$  & ${y}$ & $ {\bf R} $ & $ {\bf G} $ & $ {\bf R} $ &   ${y}$ & $  {\bf G}  $ & $g$ & ${\bf R} $ & ${y}$
\tabularnewline
\hline      
${\sharp 7}$   &    $ 10 $  &
$ {\bf R}$ & $g$ & $g$ & $ {\bf R} $ & $ {\bf Y}$ &
$ {\bf R}$ & ${\bf Y}$ & $ {\bf Y} $ & $ {\bf R} $ & $ {\bf R} $ &
$ {\bf Y}$ & ${\bf R}$ & ${y}$ & $g$ & $g$
\tabularnewline
\hline      
${\sharp 8}$   &    $ 11 $  &
$ {\bf R} $ & $ {\bf R} $ & $ {\bf G}$ & $r$ & $ {\bf Y}$ &
$ {\bf G}$  & $  {\bf Y}$ & $ {\bf G} $ & $  {\bf Y}  $ & $ {\bf Y} $ &
$ {\bf Y}$ & $g$ & $  {\bf R}$ & $r$ & $g$  
\tabularnewline
\hline      
${\sharp 9}$  &   $ 11 $  &
$ {\bf R} $ & $ {\bf R} $ & $ {\bf G}$ & $ {\bf Y}$ & $ {\bf G}$ &
$ {\bf Y}$  & ${y}$ & $ {\bf R} $ & $  {\bf G}  $ & $ {\bf G} $ &
$r$ & $ {\bf Y}$ & ${\bf G}$ & ${y}$ & $r$  
\tabularnewline
\hline      
${\sharp 10}$  &      $ 11 $  &
$ {\bf R} $ & $ {\bf G} $ & $ {\bf R}$ & $ {\bf Y}$ & $ {\bf G}$ &
$ {\bf R}$  & ${y}$ & $ {\bf G} $ & $  {\bf R}  $ & $ {\bf Y} $ &
$ {\bf Y} $ & $ {\bf G}$ & $g$ & $r$ & ${y}$  
\tabularnewline
\hline      
${\sharp 11}$   &    $ 11 $  &
$ {\bf R} $ & $ {\bf G} $ & $ {\bf G}$ & $ {\bf R}$ & $ {\bf G}$ &
$ {\bf Y}$  & ${y}$ & $ {\bf G} $ & $  {\bf R}  $ & $ {\bf Y} $ &
$ {y}$ & ${y}$ & $ {\bf R}$ & $ {\bf G} $ & $r$  
\tabularnewline
\hline      
${\sharp 12}$   &   $ 11 $  &
$ {\bf R} $ & $ {\bf R} $ & $ {\bf G}$ & $ {\bf Y}$ & $ {\bf Y}$ &
$ {\bf G}$  & $r$ & $ {\bf Y} $ & $  {\bf Y}  $ & $ {\bf Y} $ &
$ {\bf G} $ & $g$ & $ {\bf R}$ & $r$ & $g$  
\tabularnewline
\hline      
${\sharp 13}$    &    $ 11 $  &
$ {\bf R} $ & $ {\bf R} $ & $ {\bf G}$ & $ {\bf G}$ & $ {\bf Y}$ &
$r$  & $  {\bf Y} $ & $ {\bf Y} $ & $  {\bf Y}  $ & $ {\bf Y} $ &
$ {\bf G} $ & $g$ & $ {\bf R}$ & $r$ & $g$  
\tabularnewline
\hline      
${\sharp 14}$  &   $ 11 $  &
$ {\bf R} $ & $ {\bf R} $ & $ {\bf G}$ & $ {\bf G}$ & $r $&
${\bf Y} $  & ${\bf Y} $ & $ {\bf Y} $ & $  {\bf Y}  $ & $ {\bf Y} $ &
$ {\bf G} $ & $g$ & $ {\bf R}$ & $r$ & $g$
\tabularnewline
\hline      
${\sharp 15}$    &  $ 11 $  &
$ {\bf R} $ & $ {\bf G} $ & $ {\bf R}$ & $ {\bf Y}$ & $ {\bf Y}$&
${\bf G} $  & $g$ & $ {\bf R} $ & $  {\bf G}  $ & $ {\bf R} $ &
${y}$ & $g$ & $ {\bf R}$ & ${y}$ & $ {\bf Y}$ 
\tabularnewline
\hline  
 ${\sharp 16}$  &     $ 11 $  &
$ {\bf R} $ & $ {\bf R} $ & $ g$ & $ {\bf Y}$ & ${y}$&
${y}$  & ${y}$ & $ {\bf R} $ & $  {\bf G}  $ & $ {\bf G} $ &
${\bf R}$ & $ {\bf G}$ & $ {\bf R}$ & $ {\bf R}$ & $ {\bf G}$ 
\tabularnewline
\hline      
${\sharp 17}$  &    $ 11 $  &
$ {\bf R} $ & $ {\bf G} $ & $ {\bf R} $ & $ {\bf G}$ & $ {\bf Y} $&
$ {\bf Y} $  & $g$ & $ {\bf R} $ & $  {\bf G}  $ & $ {\bf R} $ &
${y}$ & $g$ & $ {\bf R}$ & ${y}$ & $ {\bf Y}$ 
\tabularnewline
\hline      
${\sharp 18}$ &     $ 11 $  &
$ {\bf R} $ & $ {\bf G} $ & $ {\bf G} $ & $ {\bf Y}$ & $ {\bf G} $&
$ {\bf R} $  & ${y}$ & $ {\bf R} $ & $  {\bf Y}  $ & $ {\bf G} $ &
$ {\bf Y}$ & $ {\bf G}$ & $ r$ & ${y}$ & $ r$ 
\tabularnewline
\hline      
${\sharp 19}$ &   $ 11 $  &
$ {\bf R} $ & $ {\bf G} $ & $ {\bf R} $ & $ {\bf Y}$ & $ {\bf G} $&
$ {\bf R} $  & ${y}$ & $ {\bf G} $ & $  {\bf R}  $ & $ {\bf Y} $ &
$r$ & $ {\bf Y}$ & $ {\bf Y}$ & $g$ & $g$ 
\tabularnewline
\hline      
${\sharp 20}$   &   $ 11 $  &
$ {\bf R} $ & $ {\bf R} $ & $ {\bf G} $ & $ {\bf Y}$ & $ {\bf G} $&
$ {\bf Y} $  & ${y}$ & $ {\bf G} $ & $  {\bf Y}  $ & $ {\bf R} $ &
${y}$ & $r$ & $ {\bf G}$ & $g$ & ${\bf R}$ 
\tabularnewline
\hline      
${\sharp 21}$   &  $ 11 $  &
$ {\bf R} $ & $ {\bf G} $ & $ {\bf G} $ & $ {\bf G}$ & $ {\bf Y} $&
$ {\bf R} $  & ${y}$ & $ {\bf R} $ & $  {\bf Y}  $ & $ {\bf G} $ &
$ {\bf G}$ & $ {\bf Y} $ & $r$ & ${y}$ & $r$ 
\tabularnewline
\hline      
${\sharp 22}$  &   $ 11 $  &
$ {\bf R} $ & $ {\bf R} $ & $ {\bf G} $ & $ {\bf Y}$ & $ {\bf G} $&
$ {\bf Y} $  & $r$ & $ {\bf Y} $ & $  {\bf Y}  $ & $ {\bf Y} $ &
$ {\bf G}$ & $g$ & ${\bf R}$ & $r$ & $g$ 	
\tabularnewline
\hline      
${\sharp 23}$   &  $ 13 $  &  
$ {\bf R} $ & $ {\bf G} $ & $r$ & $ {\bf R}$ & $ {\bf Y} $&
$ {\bf G} $  & ${\bf Y} $ & $ {y}$ & ${\bf R}  $ &
$ {\bf G}$ & $ {\bf G} $ & ${\bf Y}$ & $ {\bf R} $ & $  {\bf Y} $ 	& 
$ {\bf G} $
\tabularnewline
\hline      
${\sharp 24}$  &  $ 11 $  &  
$ {\bf R} $ & $ {\bf G} $ &  $ {\bf G}$ & $ {\bf Y} $&
$ {\bf G} $  & ${\bf Y} $ &  ${\bf R}$ &  ${\bf R}$ &
$g$ & $ {\bf G} $ & ${y}$ & ${y}$ & $  {\bf R}$&${\bf Y}$ & $r$
\tabularnewline
\hline      
${\sharp 25}$    &    $ 11 $  &   
$ {\bf R} $ & $ {\bf R} $ & $ {\bf G}$ & $ {\bf G} $&
$ {\bf Y} $  & ${\bf Y} $ &  ${\bf Y} $ &  ${\bf G}  $ &
$ {\bf Y}$ & $ r$ & ${y}$ & $ {\bf G} $ & $ r$ 	&   ${\bf R}$ & $g$
\tabularnewline
\hline      
${\sharp 26}$ &    $ 11 $  &   
$ {\bf R} $ & $ {\bf R} $ & $ {\bf G}$ & $ {\bf Y} $&
$ {\bf Y} $  & ${\bf Y} $ &  ${\bf G}$  & $r$ &  ${\bf Y}  $ &
$ {\bf G}$ &  $ {\bf Y} $ & $g$ &  $r$ & ${\bf R}$ & $g$
\tabularnewline
\hline      
${\sharp 27}$  &    $ 13 $  &   
$ {\bf R} $ & $ {\bf G} $ & $ {\bf G}$ & $ {\bf Y} $&
$ r$  & ${\bf Y} $ &  $ {\bf R} $ &  ${\bf G}  $ &
$ {\bf G}$ & ${y}$&   ${\bf R}$ & ${\bf Y}$ &${\bf R}$ & ${\bf Y}$ & $ {\bf G}$
\tabularnewline
\hline    
${\sharp 28}$    &    $ 11 $  &   
$ {\bf R} $ & $ {\bf G} $ & $ {\bf R}$ & $ {\bf Y} $&
 ${\bf Y} $ &  $ {\bf G} $ & ${y}$ &  ${\bf G}  $ &
$r$ &  ${\bf Y}$ & ${\bf G}$ &$r$ & ${\bf G}$ & $ {\bf R}$ & ${y}$
\tabularnewline
\hline      
${\sharp 29}$    &   $ 11 $  &   
$ {\bf R} $ & $ {\bf G} $ & $ {\bf R}$ & $ {\bf Y} $&
 ${\bf Y} $ &  $ {\bf G} $ &  ${y}$ &
$ {\bf G}$ & $r$&   ${\bf Y}$ & ${\bf G}$ &$r$ & ${\bf G}$ & $ {\bf R}$ & ${y}$
\tabularnewline
\hline      
${\sharp 30}$ &   $ 11$  &   
$ {\bf R} $ & $ {\bf G} $ & $ {\bf R}$ & $ {\bf G} $&
 ${\bf Y} $ &  ${\bf Y} $& $r$ &  ${\bf Y}  $ &
$ {\bf Y}$ &${\bf R}$ & $g$ &${\bf G}$ & ${\bf Y}$ & $r$ & $g$
\tabularnewline
\hline     
${\sharp 31}$   &  $ 11$  &  
$ {\bf R} $ & $ {\bf G} $ & $ {\bf R}$ & 	$ {\bf Y}$ & 	$ {\bf Y} $ &  
$ {\bf R} $ & 	$ {\bf G} $ & $ {\bf Y} $ & $g$ & 	$ {\bf G} $ &  	
$ r$ & ${y}$ & $ {\bf Y} $ &  $ {\bf R} $ & $g$ 	
\tabularnewline
\hline      
${\sharp 32}$  &     $ 11$  &  
$ {\bf R} $ & $ {\bf G} $ & $ {\bf Y} $ & $ {\bf Y} $ & $ {\bf Y} $ &
$ {\bf G} $ & $ {\bf R} $ & $ {\bf R} $ & ${y}$ &   $ {\bf R} $ & 
$g$ & $g$ & $ {\bf Y} $ & $r$ & $ {\bf G} $ 
\tabularnewline
\hline      
${\sharp 33}$ &   $ 13$  &  
$ {\bf R} $ & $ {\bf G} $ & $g$ & $ {\bf Y} $ &  $ {\bf Y} $ & $ {\bf R} $ & 
$ {\bf G} $ & $ {\bf G} $ & $ {\bf R} $ & $ {\bf Y} $ & $ {\bf Y} $ & 	
$ {\bf R} $ & 	$ {\bf G} $ &  $ {\bf R} $ &  ${y}$
\tabularnewline
\hline  
${\sharp 34}$    &   $ 11$  & 
${\bf R}$ & ${\bf G}$ & ${\bf R}$ & ${\bf G}$ 	& ${\bf R}$& ${\bf Y}$
& ${\bf Y}$ & ${\bf G}$ & ${\bf R}$ & $r$ & ${y}$ & ${y}$ &  ${\bf Y}$  
& ${\bf G}$ & $g$
\tabularnewline
\hline
${\sharp 35}$  &   $ 13$  & 
${\bf R}$ & ${\bf G}$	& ${\bf G}$ & ${\bf R}$	& ${\bf Y}$ & $r $ &  ${\bf Y}$ 	& ${\bf Y}$ & ${\bf G}$	& ${\bf Y}$	& ${\bf R}$	
& ${\bf G}$ & ${\bf G}$	& ${\bf R}$	& ${y}$ 	
\tabularnewline
\hline
${\sharp 36}$  &    $ 11$  & 
${\bf R}$ &${\bf G}$ & ${\bf R}$ & ${\bf G}$ & ${\bf R}$ & ${\bf Y}$ & ${\bf Y}$ & 	${\bf R}$ &  ${\bf G}$ & $r$ & ${y}$ & ${\bf Y}$ &  ${y}$ & $g$ & ${\bf G}$ 	
\tabularnewline
\hline
${\sharp 37}$  &  $ 11$  & 
${\bf R}$ & ${\bf G}$ & ${\bf G}$ & ${\bf G}$ &  ${\bf Y}$ &  ${\bf R}$ & ${\bf Y}$ & 	${\bf R}$ & ${y}$ & 	${\bf G}$ &  ${\bf G}$ & 	${\bf Y}$ & $r$ & ${y}$ & $r$
\tabularnewline
\hline
  ${\sharp 38}$  &   $ 11$  & 
${\bf R}$ & ${\bf R}$ & ${\bf G}$ & ${\bf Y}$ & ${\bf G}$ &  ${\bf Y}$ & ${\bf Y}$ & ${\bf G}$ & 	$r$	& ${\bf R}$ &$g$ & $r$ & ${\bf G}$ & ${y}$  & ${\bf Y}$   	
\tabularnewline
\hline
${\sharp 39}$ & $ 11$  & 
${\bf R}$ & ${\bf R}$ &  ${\bf G}$ & ${\bf G}$ & ${\bf R}$ & ${\bf R}$ & ${\bf G}$ & ${\bf Y}$ & ${\bf Y}$ & ${\bf Y}$ & ${y}$ & $r$ & $g$ & ${y}$& ${\bf G}$ 
\tabularnewline
\hline
${\sharp 40}$  &   $ 11$  &   
${\bf R}$ & ${\bf G}$ & ${\bf G}$ & ${\bf Y}$ & ${\bf G}$ & ${\bf Y}$ & ${\bf R}$ & 	${\bf  Y}$ & ${\bf R}$ & ${\bf R}$ & ${y}$ & ${y}$& $r$  & ${\bf G}$ & $g$
\tabularnewline
\hline    
\end{tabular}
\end{table}
%  %  
\begin{table}[h!]
\caption{The  largest minimal defining set of ${\rm STS}(15)$}
\label{Ldefining sts15 NO2}
\begin{tabular}{|@{\hspace{3pt}}c@{\hspace{3pt}}|c@{\hspace{3pt}}||@{\hspace{3pt}}c@{\hspace{3pt}}|@{\hspace{3pt}}c@{\hspace{3pt}}|@{\hspace{3pt}}c@{\hspace{3pt}}|c@{\hspace{3pt}}|@{\hspace{3pt}}c@{\hspace{3pt}}|@{\hspace{3pt}}c@{\hspace{3pt}}|@{\hspace{3pt}}c@{\hspace{2pt}}|@{\hspace{2pt}}c@{\hspace{2pt}}|@{\hspace{2pt}}c@{\hspace{2pt}}|@{\hspace{2pt}}c@{\hspace{2pt}}|@{\hspace{2pt}}c@{\hspace{2pt}}|@{\hspace{2pt}}c@{\hspace{2pt}}|@{\hspace{2pt}}c@{\hspace{3pt}}|@{\hspace{3pt}}c@{\hspace{3pt}}|@{\hspace{3pt}}c@{\hspace{3pt}}|@{\hspace{3pt}}c@{\hspace{3pt}}|@{\hspace{3pt}}c@{\hspace{3pt}}|@{\hspace{3pt}}c@{\hspace{3pt}}}
\hline
 No: &   ${\mathcal{D}}({\rm STS}(15),3)$   &   
$0$ & $1$ & $2$ & $3$ & $4$ & $5$ & $6$ & $7$ & $8$ & $9$ & $10$ & $11$  & $12$ & $13$  & $14$
\tabularnewline
\hline   
${\sharp 41}$  &  $ 13$  &     
${\bf R}$ & ${\bf G}$ & ${\bf G}$ &${\bf Y}$ & ${\bf R}$ & ${\bf Y}$ & $r$ & ${\bf Y}$ & ${\bf G}$ & ${\bf Y}$ & 	${\bf R}$ &   ${y}$ & ${\bf R}$ & ${\bf G}$ & ${\bf G}$ 	
\tabularnewline
\hline
${\sharp 42}$  &  $ 13$  &    
${\bf R}$ & ${\bf G}$ & ${\bf R}$ & ${\bf R}$ & ${\bf G}$ & ${\bf R}$ & ${y}$ &   	${\bf G}$ & ${\bf R}$ & ${\bf Y}$ & ${\bf Y}$ &  $g$ & ${\bf G}$ & ${\bf Y}$ & 
${\bf Y}$   	
\tabularnewline
\hline
${\sharp 43}$  &  $ 11$  &   
${\bf R}$ & ${\bf G}$ & ${\bf R}$ & ${\bf R}$ & ${\bf G}$ & ${\bf Y}$ &  ${\bf Y}$  & ${\bf G}$ &  ${\bf R}$ &  ${y}$ & $r$ & ${\bf G}$ &  	
$g$ & ${y}$ &  ${\bf Y}$ 
\tabularnewline
\hline
${\sharp 44}$  &   $ 11$  &   
${\bf R}$ & ${\bf G}$ & ${\bf R}$ & ${\bf Y}$ & ${\bf Y}$ &  ${\bf G}$ &   ${\bf R}$ & 	${\bf Y}$ & ${y}$ &  ${\bf G}$ & 	$g$ & ${y}$ & $r$ &  	${\bf G}$ & ${\bf R}$  	
\tabularnewline
\hline
${\sharp 45}$  &    $ 11$  &   
${\bf R}$ &  ${\bf G}$ & ${\bf R}$ & ${\bf Y}$ & ${\bf Y}$ & 	${\bf G}$ &  ${\bf R}$ &  ${\bf Y}$ & ${y}$ &  ${\bf G}$ &  	${\bf Y}$ & $r$ & ${\bf G}$ &  $g$ & $r$
\tabularnewline
\hline
${\sharp 46}$  &     $ 13$  & 
${\bf R}$ & ${\bf G}$ & ${\bf R}$ &	${\bf R}$ &${\bf G}$ & ${\bf R}$ & ${y}$ & ${\bf G}$ & ${\bf R}$ & 	${\bf Y}$ &	${\bf Y}$ & $g$ & 	
${\bf G}$ & ${\bf Y}$ & ${\bf Y}$	
\tabularnewline
\hline
${\sharp 47}$   &    $ 13$  &  
${\bf R}$ & ${\bf R}$ &	${\bf G}$ &	${\bf G}$ & ${\bf Y}$ &${\bf R}$ & 	${\bf G}$ & 	${\bf Y}$ & ${y}$ & ${\bf Y}$ & ${\bf Y}$ &
${\bf R}$ &	${\bf G}$ &${\bf G}$ & $r$  	
\tabularnewline
\hline
${\sharp 48}$ &  $ 13$  &  
${\bf R}$ & ${\bf G}$ &${\bf G}$ &${\bf Y}$ & ${\bf G}$ & ${\bf Y}$ &${\bf R}$ & 	${\bf Y}$ &${\bf R}$ &	${\bf R}$ &	${\bf Y}$ &	
${y}$ &  ${\bf R}$ & 	$g$ & ${\bf G}$ 
 \tabularnewline
\hline
${\sharp 49}$  &  $ 13$  &  
${\bf R}$ &	${\bf R}$ & ${\bf G}$	& ${\bf G}$ & ${\bf Y}$	& ${\bf R}$ & ${\bf G}$  	& ${\bf Y}$  & ${\bf Y}$ & ${y} $  & ${\bf Y}$ & ${\bf R}$ & ${\bf G}$ & ${\bf G}$ 	
 & $r$  	
\tabularnewline
\hline
${\sharp 50}$ &   $ 13$  &  
${\bf R}$ & ${\bf G}$ & ${\bf R}$ & ${\bf G}$ 	& ${\bf Y}$ & ${\bf Y}$ & ${\bf G}$ & 	${\bf R}$ &  ${\bf G}$ & ${\bf Y}$ &  ${\bf Y}$ &  ${\bf G}$ &  	
${\bf R}$ &   ${y}$ &   	$r$  	
\tabularnewline
\hline 
${\sharp 51}$  &  $ 13$  &  
${\bf R}$ & ${\bf G}$ & ${\bf R}$ & ${\bf Y}$ & ${\bf Y}$ & ${\bf G}$& 
${\bf R}$ & ${\bf Y}$ & ${\bf Y}$ & ${\bf G}$ & $g$ &   ${\bf R}$		 	
 & ${y}$ &   ${\bf R}$ & ${\bf G}$		
\tabularnewline
\hline
${\sharp 52}$  &  $ 11$  &     
${\bf R}$ & ${\bf G}$ & ${\bf R}$ & ${\bf G}$ &	${\bf R}$ & ${\bf Y}$ & ${\bf Y}$ & ${\bf Y}$ &${\bf Y}$ & $g$ & $g$ & $r$ & ${\bf G}$ & 	${\bf R}$ & ${y}$  	
\tabularnewline
\hline
${\sharp 53}$  & $ 11$  &  
${\bf R}$ &${\bf R}$ & 	${\bf G}$ &	${\bf G}$ & ${\bf Y}$ &	 ${\bf R}$ &	${\bf G}$ & ${\bf Y}$ &  ${y}$ &  ${\bf Y}$ & 	${y} $ & $r$ &  ${\bf G}$ & 	${\bf G}$ &		$r$ 	
\tabularnewline
\hline
${\sharp 54}$  &    $ 13$  &  
${\bf R}$ & ${\bf G}$ &	$r$ & ${\bf G}$ &  ${\bf Y}$ & ${\bf R}$ &	${\bf Y}$ &		${\bf G}$ & ${\bf G}$ &	${\bf R}$ & ${y}$ & ${\bf Y}$ &	
${\bf G}$ &	${\bf Y}$ &  ${\bf R}$  	
\tabularnewline
\hline
${\sharp 55}$  &  $ 13$  &     
${\bf R}$ &${\bf R}$ & 	${\bf G}$ &${\bf G}$ &	${\bf Y}$ &  ${\bf R}$ & 	${\bf G}$ &	${\bf Y}$ &	${\bf Y}$ &	${\bf Y}$ & ${y}$ & ${\bf R}$ & 	 ${\bf G}$ &	
 ${\bf G}$	& $r$
\tabularnewline
\hline
${\sharp 56}$  &   $ 13 $  &
${\bf R}$ &  ${\bf G}$ & ${\bf G}$ &${\bf Y}$ &	${\bf R}$ &	${\bf Y}$ & ${\bf G}$ & 	${y}$ & ${\bf R}$ & ${\bf G}$ &	${\bf G}$ &	${\bf Y}$ &
$r$ &  	${\bf Y}$ & ${\bf R}$ 
\tabularnewline
\hline
${\sharp 57}$  &  $ 13$  &
${\bf R}$ & ${\bf R}$ &${\bf G}$ &	${\bf G}$ &  ${\bf Y}$ & ${\bf R}$ & ${\bf G}$ &  	${\bf Y}$ &${\bf Y}$ &	${\bf Y}$ &	${y}$ & ${\bf R}$ &	
${\bf G}$ &	${\bf G}$ & $r$ 	
\tabularnewline
\hline
${\sharp 58}$  &  $ 13$  &
${\bf R}$ &  ${\bf G}$ &${\bf G}$ &	${\bf Y}$ &	${\bf R}$ & ${\bf Y}$ &	${\bf G}$ & 	${y}$ &  ${\bf R}$ & 	${\bf G}$ & ${\bf G}$ &
${\bf Y}$ &  $r$ & ${\bf Y}$ &	${\bf R}$ 
\tabularnewline
\hline
${\sharp 59}$  &     $ 11$  &
${\bf R}$ &${\bf G}$ &${\bf G}$ & ${\bf Y}$ & ${\bf R}$ &${\bf Y}$ & ${\bf G}$ & 	${\bf R}$ & $g$ &  	${\bf G}$ & $r$ & 	${y}$ & ${y}$& ${\bf R}$ &
${\bf Y}$ 
\tabularnewline
\hline
${\sharp 60}$  &     $ 13$  &
${\bf R}$ &  ${\bf G}$ & ${\bf R}$ & ${\bf G}$ & ${\bf R}$ &	${y}$ & ${\bf R}$ & 	${\bf Y}$ & ${\bf Y}$ & 	${\bf G}$ & 	$g$ & ${\bf Y}$ & 	
${\bf Y}$ &  ${\bf R}$ & 	${\bf G}$ 
\tabularnewline
\hline
${\sharp 61}$  &   $ 11$  &	
${\bf R}$ &  ${\bf G}$ & ${\bf Y}$ &  ${\bf R}$ &
 ${\bf Y}$ &  ${\bf G}$ &	$g$ &   ${\bf R}$ &  ${\bf Y}$ &	${\bf Y}$ &  
$r$ &${\bf Y}$ &  ${\bf G}$ & $g$ & $r$   	 	
\tabularnewline
\hline
${\sharp 62}$   &     $ 11$  &	
${\bf R}$ &  ${\bf G}$ & ${\bf G}$ &  ${\bf R}$ &	${\bf Y}$ &  ${\bf G}$ & 		$g$ &  ${\bf Y}$ &  ${\bf G}$ &  	${\bf Y}$ &  $r$ & ${y}$ & 	${\bf R}$ & $r$ & ${\bf Y}$
\tabularnewline
\hline
${\sharp 63}$ &     $ 13$  &	
${\bf R}$ &  ${\bf G}$ & ${\bf Y}$ &  ${\bf G}$ &
${\bf G}$ &  ${\bf R}$ &  	${y}$ & ${\bf Y}$ &  ${\bf R}$ & 	
${\bf Y}$ &  ${\bf R}$ & ${\bf G}$ &  ${\bf G}$ &
 ${\bf Y}$ &   $r$ 
\tabularnewline
\hline	
${\sharp 64}$   &     $ 11$  & 
${\bf R}$ &  ${\bf G}$ & ${\bf Y}$ & ${\bf R}$ &  ${\bf G}$ & ${\bf Y}$ &
${\bf Y}$ &  ${\bf G}$ & ${\bf R}$ &  ${\bf Y}$ & 	${y}$ & $r$ & $g$ 
 & ${\bf G}$ &  $r$ 	
\tabularnewline
\hline	
${\sharp 65}$   &   $ 13$  & 
${\bf R}$ &  ${\bf R}$ & ${\bf G}$ &  ${\bf G}$ &  ${\bf R}$ & ${\bf R}$ &	${\bf G}$ & ${y}$ &  ${\bf R}$ &  ${\bf Y}$ & ${\bf Y}$ &  $g$ & 
${\bf G}$ &  ${\bf Y}$ & ${\bf Y}$ 	
\tabularnewline
\hline	
${\sharp 66}$   &   $ 13$  & 
${\bf R}$ &  ${\bf G}$ & $g$ & ${\bf G}$ &  ${\bf R}$ & ${y}$ & ${\bf R}$ &  ${\bf Y}$ & ${\bf Y}$ & ${\bf G}$ &  ${\bf R}$ & ${\bf R}$ &
${\bf G}$ &  ${\bf Y}$ & ${\bf Y}$  	
\tabularnewline
\hline	
${\sharp 67}$  &   $ 13$  & 
${\bf R}$ &  ${\bf G}$ & ${\bf G}$ & ${\bf R}$ &  ${\bf Y}$ & ${\bf G}$ &
${\bf Y}$ &  ${\bf Y}$ & ${\bf R}$ & 	${\bf Y}$ &  ${\bf R}$ & ${\bf Y}$ & 
$r$ & $g$ &  ${\bf G}$ 	
\tabularnewline
\hline	
${\sharp 68}$  &  $ 13$  & 
${\bf R}$ &  ${\bf R}$ & ${\bf G}$ & ${\bf R}$ &  ${\bf G}$ & ${\bf G}$ &
 ${\bf R}$ &  ${\bf Y}$ & ${\bf Y}$ &	 ${\bf Y}$ & 	${y}$ & 
${\bf Y}$ &  ${\bf G}$ & ${\bf R}$ & $g$ 
\tabularnewline
\hline	
${\sharp 69}$ &   $ 13$  & 
${\bf R}$ &  ${\bf R}$ & ${\bf G}$ &  ${\bf R}$ &  ${\bf G}$ & ${\bf G}$ &
${\bf R}$ &  ${\bf Y}$ & ${\bf Y}$ & ${\bf Y}$ &  ${\bf Y}$ &  
${y}$ & ${\bf R}$ &  ${\bf G}$  &  $g$  	 	
 \tabularnewline
\hline	
${\sharp 70}$  &  $ 13$  & 
${\bf R}$ &  ${\bf G}$ & ${\bf Y}$ &
${\bf Y}$ &  ${\bf Y}$ & ${\bf R}$ &
${\bf G}$ &  ${\bf R}$ & ${\bf G}$ & 	${\bf R}$ & ${\bf G}$ &  ${\bf Y}$ & 
${y}$ & $r$ & ${\bf G}$  	
\tabularnewline
\hline	
${\sharp 71}$ & $ 13$  & 
${\bf R}$ &  ${\bf G}$ & ${\bf G}$ &
${\bf G}$ &  ${\bf G}$ & ${y}$ &   ${\bf R}$ & $r$ &   ${\bf Y}$ &  ${\bf R}$ & ${\bf Y}$ & 	${\bf Y}$ &  ${\bf R}$ & ${\bf Y}$ &   ${\bf G}$	
\tabularnewline
\hline	
${\sharp 72}$ & $ 13$  & 
${\bf R}$ &  ${\bf G}$ & ${\bf Y}$ & ${\bf Y}$ &  ${\bf Y}$ & ${\bf G}$ &
${\bf R}$ &  ${\bf R}$ & ${\bf G}$ &  ${\bf R}$ &  ${\bf G}$ & ${\bf Y}$ &  	
 ${\bf G}$ &  $r$ & ${y}$  		
\tabularnewline
\hline	
${\sharp 73}$   &  $ 11$  &  
${\bf R}$ &  ${\bf G}$ & ${\bf R}$ &
${\bf Y}$ &  ${\bf Y}$ & ${\bf G}$ &
${\bf R}$ &  ${\bf Y}$ & ${\bf Y}$ &  	${\bf R}$ &  ${\bf G}$  &  	
 $g$ & $g$ & ${y}$ & $r$ 	  	
\tabularnewline
\hline	
${\sharp 74}$   &$ 15$  &    
${\bf R}$ &  ${\bf G}$ & ${\bf R}$ &
${\bf G}$ &  ${\bf R}$ & ${\bf Y}$ &
${\bf Y}$ &  ${\bf R}$ & ${\bf G}$ & 	${\bf R}$ &  ${\bf G}$ & ${\bf Y}$ &  	
${\bf Y}$ &  ${\bf Y}$ & ${\bf G}$   		
\tabularnewline
\hline	
${\sharp 75}$   &   $ 13$  & 
${\bf R}$ &  ${\bf G}$ & ${\bf R}$ &
${\bf G}$ &  ${\bf R}$ & ${\bf Y}$ &  	${\bf Y}$ &  ${\bf R}$   & $g$ &   		  	
${\bf R}$ &  ${\bf G}$ & ${\bf Y}$ & ${y}$  &  ${\bf Y}$ & ${\bf G}$ 		
\tabularnewline
\hline	
${\sharp 76}$   &  $ 11$  & 
${\bf R}$ &  ${\bf R}$ & ${\bf G}$ &
${\bf R}$ &  ${\bf G}$ & ${\bf G}$ &  
${\bf Y}$ &  ${\bf R}$ & 	$g$ & 
${\bf G}$ &  ${\bf Y}$ & 	$r$ & ${y}$ & ${y}$ &  ${\bf Y}$ 
\tabularnewline
\hline	
${\sharp 77}$ &  $ 15$  &   
${\bf R}$ &  ${\bf G}$ & ${\bf R}$ &
${\bf Y}$ &  ${\bf Y}$ & ${\bf G}$ &
 ${\bf R}$ &  ${\bf G}$ & ${\bf R}$ & 	
${\bf Y}$ &  ${\bf G}$ & ${\bf R}$ & 
${\bf G}$ &  ${\bf Y}$ & ${\bf Y}$ 	
\tabularnewline
\hline	
${\sharp 78}$    &  $ 11$  & 
${\bf R}$ &  ${\bf G}$ & ${\bf Y}$ &
${\bf G}$ &  ${\bf R}$ & ${\bf Y}$ &
${\bf Y}$ &  ${\bf G}$ & $r$ &   
${\bf Y}$ &  ${\bf Y}$ & $g$ & 
$r$ & $g$ & 	${\bf R}$ 
\tabularnewline
\hline	
${\sharp 79}$   &  $ 11$  & 
${\bf R}$ &  ${\bf G}$ & ${\bf Y}$ &
${\bf Y}$ &  ${\bf Y}$ & ${\bf G}$ & 	
${\bf R}$ &  ${\bf R}$ & ${\bf G}$ &  	
${\bf R}$ &	$g$ & ${y}$ & ${\bf Y}$ & $r$ &	$g$
\tabularnewline
\hline	
${\sharp 80}$ &  $ 15$  & 
${\bf R}$ &  ${\bf G}$ & ${\bf Y}$ & 
${\bf Y}$ &  ${\bf R}$ & ${\bf Y}$ &  	
${\bf R}$ &  ${\bf G}$ & ${\bf G}$ &  		
${\bf G}$ &  ${\bf G}$ & ${\bf R}$ &
${\bf Y}$ &  ${\bf R}$ & ${\bf Y}$   	
\tabularnewline
\hline	
\end{tabular}
\end{table}
%%%%%%%%%%%%%%%%%%%%%%%%%%%%%%%%%%%%%%%%%%%%%%%%%%%%%%%%%%%%%%%%%%%%%%%%
%%%%%%%%%%%%%%%%%%%%%%%%%%%%%%%%%%%%%%%%%%%%%%
\newpage
%\bibliographystyle{plain}
%\bibliography{besharatithesisref6}

\begin{thebibliography}{10}

\bibitem{MR0384579}
Claude Berge.
\newblock {\em Graphs and hypergraphs}.
\newblock North-Holland Publishing Co., Amsterdam, revised edition, 1976.
\newblock Translated from the French by Edward Minieka, North-Holland
  Mathematical Library, Vol. 6.

\bibitem{MR2368647}
J.~A. Bondy and U.~S.~R. Murty.
\newblock {\em Graph theory}, volume 244 of {\em Graduate Texts in
  Mathematics}.
\newblock Springer, New York, 2008.

\bibitem{MR2246267}
Charles~J. Colbourn and Jeffrey~H. Dinitz, editors.
\newblock {\em Handbook of combinatorial designs}.
\newblock Discrete Mathematics and its Applications (Boca Raton). Chapman \&
  Hall/CRC, Boca Raton, FL, {S}econd edition, 2007.

\bibitem{MR2011736}
Diane Donovan, E.~S. Mahmoodian, Colin Ramsay, and Anne~Penfold Street.
\newblock Defining sets in combinatorics: a survey.
\newblock In {\em Surveys in combinatorics, 2003 ({B}angor)}, volume 307 of
  {\em London Math. Soc. Lecture Note Ser.}, pages 115--174. Cambridge Univ.
  Press, Cambridge, 2003.

\bibitem{MR1953280}
A.~D. Forbes.
\newblock Uniquely 3-colourable {S}teiner triple systems.
\newblock {\em J. Combin. Theory Ser. A}, 101(1):49--68, 2003.

\bibitem{MR1961750}
A.~D. Forbes, M.~J. Grannell, and T.~S. Griggs.
\newblock On colourings of {S}teiner triple systems.
\newblock {\em Discrete Math.}, 261(1-3):255--276, 2003.
\newblock Papers on the occasion of the 65th birthday of Alex Rosa.

\bibitem{MR2469212}
C.~C. Lindner and C.~A. Rodger.
\newblock {\em Design theory}.
\newblock Discrete Mathematics and its Applications (Boca Raton). CRC Press,
  Boca Raton, FL, {S}econd edition, 2009.

\bibitem{MR1492638}
E.~S. Mahmoodian.
\newblock Some problems in graph colorings.
\newblock In {\em Proceedings of the 26th {A}nnual {I}ranian {M}athematics
  {C}onference, {V}ol. 2 ({K}erman, 1995)}, pages 215--218. Shahid Bahonar
  Univ. Kerman, Kerman, 1995.

\bibitem{MR1674887}
E.~S. Mahmoodian and E.~Mendelsohn.
\newblock On defining numbers of vertex colouring of regular graphs.
\newblock {\em Discrete Math.}, 197/198:543--554, 1999.
\newblock 16th British Combinatorial Conference (London, 1997).

\bibitem{MR1446764}
E.~S. Mahmoodian, Reza Naserasr, and Manouchehr Zaker.
\newblock Defining sets in vertex colorings of graphs and {L}atin rectangles.
\newblock {\em Discrete Math.}, 167/168:451--460, 1997.
\newblock 15th British Combinatorial Conference (Stirling, 1995).

\bibitem{MR2194763}
E.~S. Mahmoodian, Behnaz Omoomi, and Nasrin Soltankhah.
\newblock Smallest defining number of {$r$}-regular {$k$}-chromatic graphs:
  {$r\neq k$}.
\newblock {\em Ars Combin.}, 78:211--223, 2006.

\bibitem{MR0280390}
Alexander Rosa.
\newblock On the chromatic number of {S}teiner triple systems.
\newblock In {\em Combinatorial {S}tructures and their {A}pplications ({P}roc.
  {C}algary {I}nternat. {C}onf., {C}algary, {A}lta., 1969)}, pages 369--371.
  Gordon and Breach, New York, 1970.

\bibitem{MR0290993}
Alexander Rosa.
\newblock Steiner triple systems and their chromatic number.
\newblock {\em Acta Fac. Rerum Natur. Univ. Comenian. Math.}, (Publ.
  24):159--174, 1970.

\bibitem{MR1178507}
Alexander Rosa and Charles~J. Colbourn.
\newblock Colorings of block designs.
\newblock In {\em Contemporary design theory}, Wiley-Intersci. Ser. Discrete
  Math. Optim., pages 401--430. Wiley, New York, 1992.

\bibitem{MR1871828}
J.~H. van Lint and R.~M. Wilson.
\newblock {\em A course in combinatorics}.
\newblock Cambridge University Press, Cambridge, {S}econd edition, 2001.

\end{thebibliography}

\end{document}